\documentclass[12 pt]{amsart}
\usepackage{amsmath,amssymb,amsthm,amsfonts,setspace,hyperref,dsfont,yhmath,euscript}

\usepackage{color}

\newcommand{\EE}{\mathbb{E}}
\newcommand{\HH}{\mathbb{H}}

\newcommand{\NN}{\mathbb{N}}
\newcommand{\RR}{\mathbb{R}}

\newcommand{\CC}{\mathbb{C}}

\newcommand{\PP}{\mathbb{P}}
\newcommand{\TT}{\mathbb{T}}
\newcommand{\ZZ}{\mathbb{Z}}

\newcommand{\norm}[1]{\lVert#1\rVert}
\newcommand{\abs}[1]{\lvert#1\rvert}

\newtheorem{theorem}{Theorem}[section]
\newtheorem{corollary}[theorem]{Corollary}

\newtheorem{lemma}[theorem]{Lemma}
\newtheorem{proposition}[theorem]{Proposition}

\newcommand{\comment}[1]{}
\theoremstyle{definition}

\newtheorem{remark}[theorem]{Remark}

\numberwithin{equation}{section}

\numberwithin{equation}{section}

      \makeatletter
      \def\@setcopyright{}
      \def\serieslogo@{}
      \makeatother

\def\I{{I}}

\def\c{\mathcal{C}}
\def\R{\mathbb R}
\def\Q{\mathbb Q}
\def\C{\mathbb C}
\def\Z{\mathbb Z}
\def\N{\mathbb N}
\def\T{\mathbb T}

\def\dist{\text{dist}}
\def\Id{\text{Id}}
\def\e{\epsilon}
\def\bv{\mathbf v}

\def\a{\alpha}

\def\b{\beta}

\def\hk{K}

\def\p{\phi}
\def\P{\Phi}
\def\q{\eta}

\def\A{\mathcal{A}}
\def\B{\mathcal{B}}
\def\F{\mathcal{F}}
\def\H{\mathcal{H}}
\def\HH{\EuScript{H}}

\def\E{{E}}
\def\hE{{\hat E}}
\def\calU{\mathcal{U}}
\def\EE{\mathcal{E}}
\def\I{{I}}
\def\M{\mathcal{M}}
\def\P{{\Phi}}
\def\PP{\mathcal{P}}
\def\U{{\tilde V}}
\def\UU{{\tilde{\mathcal{V}}}}
\def\V{{V}}
\def\VV{\mathcal{V}}

\def\vv{{\mathcal{V}}}
\def\L{\mathcal{L}}

\def\loc{\text{loc}}

\makeatletter
\renewcommand*\env@matrix[1][*\c@MaxMatrixCols c]{%
  \hskip -\arraycolsep
  \let\@ifnextchar\new@ifnextchar
  \array{#1}}
\makeatother
\begin{document}
\title[Local rigidity]
{Smooth local rigidity for hyperbolic toral automorphisms}

\author[Boris Kalinin$^1$ \and Victoria Sadovskaya$^2$ \and Zhenqi Jenny Wang$^3$ ]{Boris Kalinin$^1$ \and Victoria Sadovskaya$^2$ \and Zhenqi Jenny Wang$^3$ }

%\subjclass[2003]{}

\address{Department of Mathematics, The Pennsylvania State University, University Park, PA 16802, USA.}
\email{kalinin@psu.edu, sadovskaya@psu.edu}

\address{Department of Mathematics\\
        Michigan State University\\
        East Lansing, MI 48824,USA}
\email{wangzq@math.msu.edu}


\thanks{$^1$  Supported in part by Simons Foundation grants 426243 and 855238}
\thanks{$^2$ Supported in part by NSF grant DMS-1764216}
\thanks{$^3$ Supported in part by NSF grant DMS-1845416}

\begin{abstract}
We study the regularity of a conjugacy $H$ between a hyperbolic toral automorphism $A$ and its smooth perturbation $f$. We show that if $H$ is weakly differentiable then it is $C^{1+\text{H\"older}}$ and, if $A$ is also weakly irreducible, then $H$ is $C^\infty$. As a part of the proof, we establish results of independent interest on H\"older continuity of a measurable conjugacy between linear cocycles over a hyperbolic system. As a corollary, we improve regularity of the conjugacy to $C^\infty$ in prior local rigidity results.

\end{abstract}

\maketitle

\section{Introduction and local rigidity results}

Hyperbolic automorphisms of tori are the prime examples of hyperbolic dynamical systems.
The action of a matrix $A \in SL(N,\Z)$ on $\R^N$ induces an automorphism of the torus
$\T^N=\R^N/\Z^N$, which we denote by the same letter. An automorphism $A$ is called {\em hyperbolic},
or {\em Anosov}, if the matrix has no eigenvalues on the unit circle. One of the key properties of hyperbolic systems is {\it structural stability}: any diffeomorphism $f$ of $\T^N$ sufficiently $C^1$ close to such an
$A$ is also hyperbolic and is topologically conjugate to $A$ \cite{A}, that is, there exists a ho\-meo\-morphism $H$ of $\T^N$, called a {\em conjugacy}, such that
\begin{equation} \label{IntrConj}
A\circ H= H \circ f.
\end{equation}
Any two conjugacies differ by an affine automorphisms of $\T^N$ commuting with $A$ \cite{W},
and hence have the same regularity. Although $H$ is always bi-H\"older continuous, it is usually not even $C^1$, as there are various obstructions to smoothness.
This is in sharp contrast with rigidity for actions of larger groups, where often any perturbation, or even
any smooth action, is $C^\infty$ conjugate to an algebraic model.
%For a single hyperbolic system smooth classification was obtained under (strong) additional assumptions of smoothness of foliations and an invariant rigid geometric structure.

%The proof techniques for higher rank actions are also quite different.
%While there was considerable recent progress in this area,
% the techniques and methods used for group actions usually do not apply to the classical case of a single system.
 %the questions and methods are quite different.

In the classical case of a single system, the problem of establishing smoothness of the conjugacy from some weaker assumptions
has been extensively studied. It is often described as rigidity, in the sense that
weak equivalence implies strong equivalence.
%  \cite{L0,LM, L1,L2,KS03,L3,KS09,G,GKS11,G17,SY,GKS20}.

In dimension two, definitive results were obtained in \cite{L0,LM, L1}.
For hyperbolic automorphisms of $\T^2$, and more generally for Anosov diffeomorphisms  of $\T^2$, $C^\infty$ smoothness of the conjugacy was obtained from absolute continuity of $H$ and from equality of Lyapunov exponents of $A$ and $f$ at the periodic points.

The case of higher dimensional systems is much more complicated. In particular, the problem of the exact level of regularity of $H$ is subtle: for any $k \in \N$ and any $N\ge 4$ there exists a reducible hyperbolic automorphism $A$ of $\T^N$ and its analytic perturbation $f$ such that the conjugacy  is  $C^k$ but is not $C^{k+1}$ \cite{L1}. We recall that $A$ is {\em reducible}\, if it has a nontrivial rational invariant subspace or, equivalently, if its characteristic polynomial is reducible over $\Q$.
%But there is $k(L)$ ...

The two-dimensional results were extended in two directions. First, $C^\infty$ conjugacy
was obtained for systems that are conformal on full stable and unstable subspaces under various periodic data assumptions which ensured that the perturbed system is also conformal \cite{L2,KS03,L3,KS09}.
Second, for some classes of irreducible $A$, equality of Lyapunov exponents or similarity of the periodic data were shown to imply $C^{1+\text{H\"older}}$ smoothness of $H$ \cite{GG,G,GKS11,SY,GKS20,dW}.
Irreducibility of $A$ is necessary for these results \cite{L1,L2,G}.
Low smoothness of $H$ is due to the method of the proof, which establishes regularity
of $H$ along natural one or two-dimensional $f$-invariant foliations of $\T^N$, whose leaves are typically only $C^{1+\text{H\"older}}$ smooth.  Nevertheless, Gogolev conjectured in \cite{G} that the regularity of $H$
should be close to that of $f$, and in particular if  $f$ is $C^\infty$ then so is $H$.
Until now, the only progress on higher regularity of $H$, outside of the conformal setting, was obtained for automorphisms of $\T^3$  with real spectrum in \cite{G17}. We refer to \cite{KSW22} for a more detailed account of questions and developments in the area of local rigidity.
\vskip.1cm

In this paper we establish general results on bootstrap of regularity of the conjugacy.
We show that for {\it any}\, hyperbolic automorphism $A$, if $H$ is weakly differentiable in a certain sense
then it is $C^{1+\text{H\"older}}$ and, if in addition $A$ is {\it weakly irreducible}, then $H$ is $C^\infty$.
We introduce and discuss the weak irreducibility property, which is  weaker than irreducibility and holds for some $A$ with Jordan blocks.
Our methods are different from those in the previous local rigidity results. In particular, we prove smoothness of $H$ without establishing it first along invariant foliations. For the $C^\infty$ smoothness of the conjugacy, we use a KAM type iterative scheme. This approach is novel  in the setting of hyperbolic systems, and it is substantially different from previous applications of KAM, as we discuss below. As a corollary, we improve the regularity of $H$ from $C^{1+\text{H\"older}}$ to $C^\infty$ in the previous local rigidity results for the irreducible case.
\vskip.1cm

Now we formulate our main results. We denote by $W^{1,q}(\T^N)$ the Sobolev space of $L^q$
functions with $L^q$ weak partial derivatives of first order.  We note that Lipschitz functions are in
$W^{1,\infty}(\T^N)$.

The first result holds for an
arbitrary hyperbolic automorphism without any irreducibility assumption.
We recall that while $H$ satisfying \eqref{IntrConj} is not unique,  there is a unique {\em conjugacy $C^0$ close to the identity}. This is $H$ in the homotopy class
of the identity with $H(p)=0$, where $p$ is the fixed point of $f$ closest to $0$.

\begin{theorem} \label{HolderConjugacy}
Let $A$ be a hyperbolic automorphism of $\T^N$ and let $f$ be a $C^{1+\text{H\"older}}$ diffeomorphism of $\T^N$ which is $C^1$ close to $A$. Suppose that for some conjugacy $H$
between $f$ and $A$, either $H$ or $H^{-1}$ is in  $W^{1,q}(\T^N)$ with $q>N$.
Then $H$ is a $C^{1+\text{H\"older}}$ diffeomorphism.
\vskip.1cm

More precisely, there is a constant $\beta_0=\beta_0(A)$, $0<\beta_0\le1$, so that for any
 $0<\beta' <\beta_0$ there exist constants $\delta>0$ and $K>0$ such that for any $0<\beta \le \beta'$
the following holds.

For any ${C^{1+\beta}}$ diffeomorphism $f$ with $\|{f-A}\|_{C^{1}}<\delta$, if some conjugacy
between $A$ and $f$, or its inverse, is in $W^{1,q}(\T^N)$, $q>N$, then any conjugacy  is  a
$C^{1+\beta}$ diffeomorphism. Moreover, for the conjugacy $H$ that is $C^0$ close to the identity,
\begin{equation} \label{C1H est}
\|{H-I}\|_{C^{1+\beta}}\leq K\|{f-A}\|_{C^{1+\beta}}.
\end{equation}
\end{theorem}
%\newpage

\begin{remark}
The assumption of being in $W^{1,q}$ with $q>N$ in this and in the next theorem can be replaced
with a slightly weaker one that we actually need for the proof:  either $H^{-1}$
is in $W^{1,1}$ and its Jacoby matrix of partial derivatives is invertible and gives the differential
of $H^{-1}$ for Lebesgue almost every point of $\T^N$, or the same holds for $H$ and $f$ preserves
an absolutely continuous probability measure.
\end{remark}

%{\color{red} $K$ is from second paragraph. check if $K$ is a constant only dependent on $A$}

In the next theorem we obtain $C^\infty$ smoothness of the conjugacy assuming that $A$ is weakly irreducible, which we define as follows. Let $\R^N=\oplus_{\rho_i} E^i$ be the splitting where $E^i$ is the sum of generalized eigenspaces of $A$ corresponding to the eigenvalues of modulus $\rho_i$, and let $\hat E^i=\oplus_{\rho_j\neq \rho_i} E^j.$
 We say that $A$ is \emph{weakly irreducible} if each $\hat E^i$ contains no nonzero elements of $\Z^N$.
This condition is  weaker than irreducibility, see Section \ref{Weak irred} for details.

\begin{theorem}\label{th:4} Let $A $ be a weakly irreducible hyperbolic automorphism of $\T^N$.
Then there is a constant $\ell=\ell(A)\in\NN$ so that for any $C^\infty$ diffeomorphism $f$   which is $C^{\ell}$ close to $A$ the following holds. If for some conjugacy $H$ between $f$ and $A$ either $H$ or $H^{-1}$ is in the Sobolev space $W^{1,q}(\T^N)$ with $q>N$, then any conjugacy between $f$ and $A$ is a $C^\infty$ diffeomorphism. \end{theorem}

The constant $\ell=\ell(A)$ is chosen sufficiently large to satisfy the inequalities \eqref{for:78}.

\vskip.2cm

Applying Theorem \ref{th:4} we improve the regularity of the conjugacy from $C^{1+\text{H\"older}}$ to $C^\infty$ in the strongest local rigidity results for irreducible toral setting \cite{GKS11,GKS20}:

\begin{corollary} \label{local PD}
Let $A:\T^N\to\T^N$ be an irreducible Anosov automorphism
such that no three of its eigenvalues have the same modulus.
Let  $f$ be a $C^\infty$ diffeomorphism which is $C^{\ell}$-close to $A$ such that the
derivative $D_pf^n$ is conjugate to $A^n$ whenever $p=f^n(p)$.
Then $f$ is $C^{\infty}$ conjugate to $A$.
\end{corollary}

\begin{corollary} \label{local LS}
Let $A:\T^N\to\T^N$ be an irreducible Anosov automorphism such that no three of
its eigenvalues have the same modulus and there are no pairs of eigenvalues
of the form $\lambda, -\lambda$ or $i\lambda, -i\lambda$, where $\lambda\in \R$.
Let  $f$ be a volume-preserving $C^\infty$ diffeomorphism of $\T^N$ sufficiently
$C^{\ell}$-close to $A$. If the Lyapunov exponents of $f$ with respect to the volume
are the same as the Lyapunov exponents of $A$, then $f$ is $C^{\infty}$ conjugate to $A$.
\end{corollary}

Now we briefly discuss our approaches.
To prove Theorem \ref{HolderConjugacy}, we first establish results of independent interest
on H\"older continuity of a measurable conjugacy between linear cocycles over a hyperbolic system.
These results are formulated and discussed in Section \ref{cocycles statements}.
In the proof of the theorem we apply them to the conjugacy $DH$ between the derivative cocycles
$Df$ and $A$.  The methods used yield only H\"older continuity of the conjugacy between the derivative cocycles and hence only $C^{1+\beta}$ regularity of $H$. 

We note, however, that existence of {\it some} H\"older conjugacy between
the derivative cocycles $Df$ and $A$ does not imply in general that $H$ is $C^1$.
Indeed, if all eigenvalues of $A$ are simple with distinct moduli, then conjugacy of $D_pf^n$ and $A^n$, whenever $p=f^n(p)$, always gives H\"older conjugacy of the cocycles, but $H$ may not be $C^1$ if $A$ is reducible. 
\vskip.1cm

 Our approach to proving Theorem \ref{th:4} is different from  prior work on
this problem. We abandon the geometric arguments which use invariant foliations.
Instead, we introduce a new method which combines exponential mixing of the
unperturbed system with a KAM type iterative scheme. KAM methods have been extensively 
used to study local rigidity, primarily for elliptic systems, such as Diophantine translations of a torus. 
These systems are very different from the hyperbolic ones that we are considering.
Closest to our setting, KAM techniques were used in \cite{Damjanovic4} to prove $C^\infty$ local rigidity 
for some $\ZZ^2$ actions by partially hyperbolic toral automorphisms. However,
our approach is markedly different from the existing work in both main ingredients
of the KAM method: a detailed study of the linearized conjugacy equation, and setup and convergence of the  iterative process.  In particular, the linearized conjugacy equation in our case is a twisted 
cohomological equation with the twist given by a hyperbolic matrix.
 In contrast to the elliptic case, this creates obstructions to solving the equation by sufficiently smooth functions. In \cite{Damjanovic4} the structure of $\ZZ^2$ action was used in an essential way to show vanishing of the obstructions. In our setting, we instead use the existence of a $C^{1+\b}$ conjugacy $H$ given by  Theorem \ref{HolderConjugacy}. However, two difficulties arise. First, higher regularity is 
needed for analyzing the obstructions (see  Lemma \ref{le:3}(iii)). 
Second, $C^{1+\b}$
functions have slow decay of Fourier coefficients. In contrast, super-polynomial decay for $C^\infty$
functions, which yields super-exponential mixing,
was crucial in obtaining suitable estimates for convergence in the KAM iteration in \cite{Damjanovic4}. One of the key innovations in our approach is using directional derivatives to ``balance'' the twist.
By analyzing differentiated equations in H\"older category we are able
to construct a $C^\infty$ approximate solution of the twisted cohomological equation
and obtain suitable estimates.
This is done in Section \ref{linearized}, see remarks after Theorem \ref{th:3} for details.
Relating Fourier coefficients of a function and its directional derivatives is the only place where
we use weak irreducibility of $A$. The last issue is that the estimate we obtain for the approximate
solution is not tame, in contrast to traditional KAM estimates \cite{Fayad}. This creates problems in establishing convergence of the iterative
procedure, which we overcome  in Section \ref{proof th:4}.

\vskip.1cm

The paper is structured as follows. In Section \ref{cocycles statements} we formulate our results
on continuity of a measurable conjugacy between linear cocycles over a hyperbolic system,
Theorems \ref{measurable conjugacy} and \ref{constant cocycle}. These theorems are proved in
Sections \ref {cocycle proofs} and \ref{proof of constant},  respectively.
In  Section \ref{section:notation}  we summarize basic notations and facts used throughout the paper.
In  Section \ref{proof HC} we prove Theorem \ref{HolderConjugacy}.
In  Section \ref{linearized} we obtain a result on solving a twisted cohomological equation over $A$,
%, which is the linearization of the conjugacy equation.
and in Section \ref{proof th:4} we complete the proof of Theorem \ref{th:4}.
%%%%%%%%%%%%%%%%%%%%%%%%%%%%%

\section{Results on continuity of conjugacy between linear cocycles} \label{cocycles statements}

In this section we consider linear cocycles over a transitive Anosov diffeomorphism $f$
of a compact connected manifold $\M$.
% Remark on hyperbolic homeomorphism?
We recall that  $f$ is {\it Anosov}\, if there exist a splitting
of the tangent bundle $T\M$ into a direct sum of two $Df$-invariant
continuous subbundles $\tilde E^s$ and $\tilde E^u$,  a Riemannian
metric on $\M$, and  continuous functions $\nu$ and $\hat\nu$  such that
\begin{equation}\label{Anosov def}
\|Df_x(\bv^s)\| < \nu(x) < 1 < \hat\nu(x) <\|Df_x(\bv^u)\|
\end{equation}
for any $x \in \M$ and any unit vectors $\,\bv^s\in \tilde E^s(x)$ and $\,\bv^u\in \tilde E^u(x)$.
The diffeomorphism is {\em transitive} if there is a point $x\in \M$ with dense orbit.
All known examples satisfy this property.

\vskip.1cm
Let $A$ be a map from $\M$ to  $GL(N,\R)$.
The $GL(N,\R)${\em -valued  cocycle over $f$ generated by }$A$
is the map $\A:\,X \times \Z \,\to GL(N,\R)$ defined  by $\,\A(x,0)=\Id\,$ and for $n\in \N$,
$$
 \A(x,n)=\A_x^n = A(f^{n-1} x)\circ \cdots \circ A(x) \;\text{ and }\;
\A(x,-n)=\A_x^{-n}= (\A_{f^{-n} x}^n)^{-1}.
$$

We say that a $GL(d,\R)$-valued cocycle $\A$ is $\beta$-H\"older continuous
if there exists a constant $c$ such that % {\tt maybe just use $\|(\A_x - \A_y\|$ here}.
$$
\,d(\A_x, \A_y) \le c\, \dist(x,y)^\beta  \quad \text{for all }x,y\in \M,
$$
where the metric $d$ on $GL(N,\R)$ is given by
$$
d (A, B) = \| A  - B \|  + \| A^{-1}  - B^{-1} \|, \quad\text{where $\|\,.\,\|$ is the operator norm.}
$$

More generally, we consider linear cocycles defined as follows.  Let  $P : \E \to \M $ be a finite
dimensional  $\beta$-H\"older vector bundle over $\M$.
A continuous {\em linear cocycle} over $f$ is a homeomorphism
$\A:\E\to\E$ such that
$$P \circ \A = f \circ P \quad\text{and\, $\A_x : \E_x \to \E_{fx}$\,
is a linear isomorphism for each $x\in \M$.}
$$
The linear cocycle $\A$ is called $\beta$-H\"older
 if $\A_x$ depends $\beta$-H\"older on $x$, with proper identification of fibers at nearby points. A detailed description of this setting is given in Section 2.2 of \cite{KS13}.
\vskip.2cm

The differential of $f$ and its restrictions to invariant sub-bundles of $T\M$, such as
 $\tilde E^s$ and $\tilde E^u$, are prime examples of linear cocycles.
\vskip.2cm

We say that a $\beta$-H\"older cocycle $\A$
over an Anosov diffeomorphism $f$ is\,
 {\em fiber bunched} if
there exist numbers $\theta<1$ and $c$  such that for all $x\in\M$ and $n\in \N$,
\begin{equation}\label{fiber bunched}
\| \A_x^n\|\cdot \|(\A_x^n)^{-1}\| \cdot  (\nu^n_x)^\beta < c\, \theta^n \;\text{ and}\quad
\| \A_x^{-n}\|\cdot \|(\A_x^{-n})^{-1}\| \cdot  (\hat \nu^{-n}_x)^\beta < c\, \theta^n,
\end{equation}
where
$ \;\nu^n_x=\nu(f^{n-1}x)\cdots\nu(x) \,\text{ and }\;
   \hat\nu^{-n}_x=(\hat \nu(f^{-n}x))^{-1}\cdots  (\hat\nu(f^{-1}x))^{-1}.$
\vskip.2cm

Let $\mu$ be an ergodic $f$-invariant measure on $\M$. We denote by  $\lambda_+(\A,\mu)$
and $\lambda_-(\A,\mu)$ the largest and smallest Lyapunov
exponents of $\A$ with respect to $\mu$ given by the Oseledets Multiplicative Ergodic Theorem.
For  $\mu$ almost all  $x\in \M$, they equal the limits
\begin{equation} \label{exponents}
\lambda_+(\A,\mu)= \lim_{n \to \infty} n^{-1} \ln \| \A_x ^n \|
\quad \text{and}\quad
\lambda_-(\A,\mu)=  \lim_{n \to \infty} n^{-1}  \ln \| (\A_x ^n)^{-1} \|^{-1}.
\end{equation}

 We say that a cocycle $\A$ {\em has one exponent}\, if
 for every $f$-periodic point $p$ the invariant measure
$\mu_p$ on its orbit satisfies $\lambda_+(\A,\mu_p)=\lambda_-(\A,\mu_p)$.
By Theorem 1.4 in \cite{K11}, this condition is equivalent to
$$
\lambda_+(\A,\mu)=\lambda_-(\A,\mu)\quad\text{for every ergodic $f$-invariant measure.}
$$
We note that if $\A$ has one exponent, then it is fiber bunched \cite[Corollary 4.2]{S15}.
\vskip.1cm

For $GL(N,\R)$ cocycles $\A$ and $\B$ over $f$, a (measurable or continuous) function $\c:\M\to GL(N,\R)$ such that
$$
  \A_x=\c(fx)\,\B_x\,\c(x)^{-1} \quad\text{for all }x\in \M
$$
is called a (measurable or continuous) {\em conjugacy} or {\em transfer map} between $\A$ and $\B$.
For linear cocycles $\A,\B : E\to E$ a conjugacy is defined similarly with $\c(x) \in GL(E_x)$.

The question whether a measurable conjugacy between two cocycles is continuous has been
studied in \cite{PaP97,Pa99,Sch99, S13,S15}. An example in \cite{PW01} shows that a measurable
conjugacy between two fiber bunched $GL(2,\R)$-valued cocycles is not necessarily continuous,
moreover, the generators of the cocycles in this example can be chosen arbitrarily close to the identity.
Continuity of a measurable conjugacy was proven for cocycles with values in a compact group
 \cite{PaP97,Pa99} and, somewhat more generally for cocycles with bounded distortion \cite{Sch99},
 for $GL(2, \R)$-valued cocycles with one exponent \cite{S13}, and for $GL(N, \R)$-valued cocycles
 such that one is fiber bunched and the other one is  uniformly quasiconformal \cite{S15}.
 The result in \cite{S13}  relied on two-dimensionality, and the uniform quasiconformality
 assumption in \cite{S15} is much stronger than having one exponent.
The next theorem establishes continuity of a measurable conjugacy between a fiber bunched cocycle and a cocycle with one exponent.

\begin{theorem} \label{measurable conjugacy}
Let $f$ be a  transitive $C^{1+\text{H\"older}}$   Anosov diffeomorphism of a compact manifold $\M$,
and let $\A$ and $\B$ be $\beta$-H\"older linear cocycles over $f$.
Suppose that $\A$ has one exponent  and $\B$ is fiber bunched.

Let  $\mu$ be an ergodic $f$-invariant  measure on $\M$ with full support and local product structure.
Then any $\mu$-measurable conjugacy between $\A$ and $\B$ is $\beta$-H\"older continuous, i.e.,
coincides with a $\beta$-H\"older continuous conjugacy on a set of full measure.

\end{theorem}

As we mentioned above, continuity of a measurable conjugacy does not hold in general if $\A$ has more than one exponent, however,  we prove it in a special case of a constant $\A$.
Moreover, we obtain an estimate of the $\beta$-H\"older constant  $\hk_\beta(\c)$ of the conjugacy $\c$ in
terms of the $\beta$-H\"older constant of $\B$.

\begin{theorem} \label{constant cocycle}
Let $f$ and $\mu$ be as in Theorem \ref{measurable conjugacy}, and
let  $\A$ is be a constant $GL(N,\R)$-valued cocycle over $f$. Then
for any H\"older continuous $GL(N,\R)$-valued cocycle $\B$ sufficiently $C^0$ close to $\A$, any $\mu$-measurable conjugacy between $\A$ and $\B$ is  H\"older continuous.
\vskip.05cm

   More specifically, there exists a constant $\beta_0(A,f)$ so that the following holds. For any $0<\beta'<\beta_0 (A,f)$   there is $\delta >0$ and $k>0$ such that for any
$0<\beta \le \beta'$ and any $\beta$-H\"older $GL(N,\R)$-valued cocycle $\B$ over $f$ with
$ \|\B_x-A\|_{C^{0}}<\delta$, any $\mu$-measurable conjugacy $\c$ between $\A$ and $\B$ is
$\beta$-H\"older and its  $\beta$-H\"older constant satisfies
\begin{equation} \label {hk C}
 \hk_\beta(\c) \le k \,  \| \c\|_{C^0} \, \hk_\beta (\B)%= k_1 \hk_\beta (\B_x -A).% \|\B_x-A\|_{C^{\b}}.
 \qquad \text {and} \qquad  \hk_\beta(\c^{-1}) \le k \,  \| \c^{-1}\|_{C^0} \, \hk_\beta (\B).
 \end{equation}
\end{theorem}

The constant $\beta_0(A,f)$ is explicitly  given by  \eqref{beta 0} in Section 5.

%%%%%%%%%%%%.  Basic notations %%%%%%%%%%%%%%%

\section{Basic notations and facts} \label{section:notation}

\subsection{Norms and H\"older constants.}

For $r\in\NN\cup\{0\}$ we use $\norm{\cdot}_{C^r}$ for the $C^r$ norm of functions with continuous derivatives up to order $r$ on $\TT^N$.

For a $\beta$-H\"{o}lder function $g$, $0<\beta\le1$, we denote its
$\beta$-H\"{o}lder constant, or H\"older seminorm, by
\begin{align*}
K_\beta(g)=\norm{g}_{C^{0,\beta}}\,\overset{\text{def}}{=}\,\sup \,\{\,|g(x)-g(y)| \, d(x,y)^{-\beta} : \; x\neq y \in \T^N\,\}\,<\,\infty.
\end{align*}
We denote by $C^{1,\beta}$ or $C^{1+\beta}$ the space of functions with $\beta$-H\"{o}lder first derivative
with norm
\begin{align*}
       \norm{f}_{C^{1+\beta}} \,\overset{\text{def}}{=} \,\norm{f}_{C^1}+K_\beta (D f).
\end{align*}
%We say that $g$ is $\beta$-H\"{o}lder along a foliation $\mathcal{V}$ if
%\begin{align*} \norm{g}_{C^{0,\beta}_{\mathcal{V}}}\, \overset{\text{def}}{=} \,\sup_{x\neq y}\,\frac{|g(x)-g(y)|}{d_{\mathcal{V}}(x,y)^\beta}\,<\,\infty \end{align*}
%for any $x,y$ along a leaf of $\mathcal{V}$, where $d_{\mathcal{V}}(x,y)$ is the infimum of lengths of smooth curves in $\mathcal{V}$ connecting $x$ and $y$.

\subsection{Invariant subspaces}

For $A\in GL(N,\R)$ let  $\rho_1 < \dots < \rho_L$ be the distinct moduli of its eigenvalues and let
\begin{equation} \label{splitL}
\R^N = \E^1 \oplus  \dots \oplus \E^L
    \end{equation}
be the corresponding $A$-invariant splitting, where $\E^i$ is the direct sum of generalized eigenspaces
corresponding to the eigenvalues with modulus $\rho_i$. We also denote
\begin{equation} \label{A_i}
{\hE}^{i}\overset{\text{def}}{=}\oplus_{\rho_j \neq \rho_i} E^{j}, \qquad \ A_i=A|_{ \E^i}:E^i\to E^i, \qquad \text{and} \qquad N_i=\dim E^i.
    \end{equation}
For the Euclidean norm on $\R^N$ there is a constant $K_A$ such that for each $i$ we have
    \begin{equation}\label{jordan}
     \norm{A_i^m}\leq K_A \,\rho_i^m \,(\abs{m}+1)^N  \quad \text{for all }  m\in\ZZ.
    \end{equation}
Also, for any $\e >0$ there is an ``adapted" inner product on $\R^N$ such that the direct sum
$\oplus E^i$ is orthogonal and  for each $1\le i\le L$,
\begin{equation} \label{rateA2}
(\rho_i-\e)^m \le \| A_i^m u \| \le  (\rho_i+\e)^m
\; \text{ for any unit vector }u\in \E^i  \text{ and any } m\in \Z.
\end{equation}
 If $A$ is hyperbolic then  $\rho_{i_0}<1 <\rho_{i_0+1} $ for some $1\le i_0<L$, and we define
 the stable and unstable subspaces of $A$ as
\begin{gather*}
%  W_{i}\overset{\text{def}}{=}\sum_{|\lambda_j|=|\lambda_i|} E_{j},\qquad
  E^{s}\overset{\text{def}}{=}\oplus_{\rho_i<1} E^{i}\qquad \text{and} \qquad  E^{u}\overset{\text{def}}{=}\oplus_{\rho_i>1} E^{i}.
  %,\quad W^{c}\overset{\text{def}}{=}\sum_{|\lambda_i|=1} E_{i}.
\end{gather*}
 %We may write $E_A^*$ for $E^*$ to emphasize the dependence on $A$.
\vskip.3cm

\subsection{Weak irreducibility.} \label{Weak irred}
Recall that  $GL(N,\Z)$ denotes the integer matrices with determinant $\pm1$.
We say that $A\in GL(N,\Z)$ is \emph{weakly irreducible} if each $\hat E^i$ contains no nonzero
elements of $\Z^N$.
 Irreducibility over $\Q$ implies weak irreducibility. Indeed, if there is a nonzero integer point $n \in \hat E^i$ then $span \{ A^m n : m\in \Z \} \subset \hat E^i$ is a nontrivial rational invariant subspace.
In fact, weak irreducibility is determined by the characteristic polynomial of $A$ as follows.

 \begin{lemma}  \label{weak irred}
A matrix $A\in GL(N,\Z)$ is weakly irreducible if and only if there is a set $\Delta \subset \R$ so that for each
irreducible over $\Q$ factor of the characteristic polynomial of $A$ the set of moduli of its roots
equals $\Delta$.
\end{lemma}

 \begin{proof}
 Let $A\in GL(N,\Z)$, let  $p_A$ be its characteristic polynomial, and let $p_A=\prod_{k=1}^K p_k^{d_k}$ be its prime decomposition over $\Q$. Then we have the corresponding splitting $\R^N = \oplus V_k$
 into rational  $A$-invariant subspaces $V_k =\ker p_k^{d_k}(A)$.  We also have the (non-rational)
 $A$-invariant splitting \eqref{splitL}, and we set $\Delta=\{\rho_1, \dots, \rho_L\}$.
We will show that $A$ is weakly irreducible if and only if $\Delta$ is the set of moduli of the roots
for each $p_k$.

If for some $\rho_i\in \Delta$ and $k\in \{1,\dots , K\}$ no root of the irreducible polynomial $p_k$ has modulus $\rho_i$, then $V_k \subset \hE^i$. Hence $A$
 is not weakly irreducible as $V_k$ is a rational subspace and hence it contains nonzero points of $\Z^N$.

 Conversely, suppose each $p_k$ has $\Delta$ as the set of moduli of its roots.
 Suppose that for some $i$ there is $0\ne n \in (\Z^N \cap \hE^i)$. Then for some $k$ its projection
 $n_k$ to $V_k$ is a nonzero rational vector. We note that $n_k \in \hE^i$ as $\hE^i= \oplus_k (\hE^i \cap V_k)$. Then
 $$W=span \{ A^m n_k : m\in \Z \}
 $$
 is a rational $A$-invariant subspace
 contained in $\hE^i \cap V_k$. Then the characteristic polynomial of $A|_W$ is a power of $p_k$
 and hence $W$ contains an eigenvector with eigenvalue of modulus $\rho _i \in \Delta$. Thus
 $W \cap E^i \ne 0$,  contradicting $W \subset \hE^i$. Thus $A$ is weakly irreducible.
 \end{proof}

It follows from the lemma that if $A$ is irreducible or weakly irreducible then the following
matrices are weakly irreducible
$$
\left(\begin{array}{cc}A & 0 \\ 0 & A \end{array}\right) \quad \text{and}\quad
\left(\begin{array}{cc}A & \text{I} \\ 0 & A \end{array}\right).
$$
These matrices are not irreducible and the latter is not diagonalizable.
%\vskip.2cm

%{\tt Can we show that if $A$ is not weakly irreducible, we have counterexamples?}

%%%%%%%%%%%%%%%%%%%%%%%%%%%%

%%%%%%%%%%%%%%  Proof of Theorem 2.1 %%%%%%%%%
%%%%%%%%%%%%%%%%%%%%%%%%%%%%%%%%%%

\section{Proof of Theorem \ref{measurable conjugacy}}
\label{cocycle proofs}

Let $f$ be a  transitive $C^{1+\text{H\"older}}$ Anosov diffeomorphism of a compact manifold $\M$,
let $\E$ be a $\beta$-H\"older vector bundle over $\M$,
and let $\F :\E \to \E$ be a $\beta$-H\"older linear cocycle over $f$.

In Section \ref{holonomy} we recall the definition and properties of holonomies for linear cocycles,
in Section \ref{twisted} we prove a preliminary results on twisted cocycles, and in Section \ref{proof of thm measurable} we give a proof of Theorem \ref{measurable conjugacy}.

%%%%%%%%%%%%%   Holonomies %%%%%%%%%%%%%%%

\subsection{Holonomies of fiber bunched linear cocycles} \label{holonomy}

The notion of {\it holonomies} for linear cocycle was introduced in \cite{BV,V}
Existence of holonomies was proved in \cite{V,ASV} under a stronger ``one-step" fiber
bunching condition and then extended to bundle setting and weaker fiber bunching \eqref{fiber bunched} in  \cite{KS13,S15}.

\begin{proposition} \label{existence of holonomies}
Let $\F$ be a $\beta$-H\"older fiber bunched linear cocycle over $(\M,f)$.
Then for every $x\in \M$ and $y\in W^s(x)$ the limit
\begin{equation}\label{hol def}
 \H^{s}_{x,y} =  \H^{\F,s}_{x,y} =\underset{n\to\infty}{\lim} \,(\F^n_y)^{-1} \circ \F^n_x,
\end{equation}
called the {\em stable holonomy,} exists and satisfies
\begin{itemize}
\item[($\H$1)] $\H^{s}_{x,y}$ is an invertible  linear map from $\E_x$ to $\E_y$;
\vskip.1cm
\item[($\H$2)]  $\H^{s}_{x,x}=\Id\,$ and $\,\H^{s}_{y,\,z} \circ \H^{s}_{x,\,y}=\H^{s}_{x,z}$,\,\,
and hence  $(\H^{s}_{x,y})^{-1}=\H^{s}_{y,x};$
\vskip.1cm
\item[($\H$3)]  $\H^{s}_{x,y}= (\F^n_y)^{-1}\circ \H^{s}_{f^nx ,\,f^ny} \circ \F^n_x\;$
for all $n\in \N$;
\vskip.1cm
\item[($\H$4)] $\| \H^{s}_{\,x,y} - \Id \,\| \leq c\cdot d (x,y)^{\beta},$
 where $c$ is independent of $x$   and $y\in W^s_{\text{loc}}(x).$\\
%and also  $\|(\A^n_y)^{-1} \circ \A^n_x - \Id \,\| \leq c\,\dist (x,y)^{\beta}\,$for every $n\in \N$.

\end{itemize}
\end{proposition}

%%%%%%%%%%%  Twisted cocycles %%%%%%%%%%%%%%

\subsection{Twisted cocycles} \label{twisted}
In this section we study the coboundary equation  over $f$ twisted by a $\beta$-H\"older
linear cocycle  $\F :\E \to \E$. We will use its main result, Proposition \ref{twist-meas},
 in the inductive process in the proof of Theorem \ref{measurable conjugacy}.
\vskip.1cm

Let $\p,\q : \M \to \E$ be sections of the bundle $\E$ over $\M$. We consider the equation
\begin{equation}\label{twisteq}
\q(x)=\p(x) + (\F_x)^{-1} (\q(fx))\quad \text{equivalently}\quad \p(x)=\q(x) - (\F_x)^{-1} (\q(fx)).
\end{equation}
Iterating \eqref{twisteq} and denoting $\F^n_x=   \F_{f^{n-1}x} \circ  \cdots  \circ \F_{fx}\circ  \F_x: \E_x \to \E_{f^{n}x}$ we obtain
$$
\begin{aligned}
 \q(x)& =\p(x) + (\F_x)^{-1} (\q(fx))= \p(x) + (\F_x)^{-1}[\p (fx)+\F_{fx} (\q(f^2x))]=...\\
  &= \p(x) + (\F_x)^{-1}(\p (fx))+\dots + (\F^{n-1}_{x})^{-1}(\p(f^{n-1}x))+  (\F_{f^{n-1}x})^{-1} (\q(f^{n}x)).
\end{aligned}
$$
Thus
\begin{equation}\label{twisteq iter}
\q(x) = \P^n(x)+ (\F_{f^{n-1}x})^{-1} (\q(f^{n}x)), \quad\text{where}
\end{equation}
$$
\P^n(x)= \p(x) + (\F_x)^{-1}(\p (fx))+\dots + (\F^{n-1}_{x})^{-1}(\p(f^{n-1}x)) \in \E_x.
$$

We say that $\F$ is {\em uniformly bounded}\, if there exists $K$
such that $\| \F_x^n \| \le K$ for all $x\in \M$ and $n\in \Z$.
A $\beta$-H\"older bounded cocycle is fiber-bunched and hence it has stable holonomies $\H^s_{x,y}:\E_x \to \E_y$ where $y \in W^s(x)$.

%%%%%
\begin{lemma} \label{twist hol}
Suppose that $\p$ is a $\beta$-H\"older section and that $\F$ is a uniformly bounded $\beta$-H\"older
cocycle.
Then for any $x\in \M$ and $y \in W^s(x)$ the following limit exists
$$
\P_{x,y}^s= \lim _{n \to \infty} (\P^n(x) - \H_{y,x}^s\P^n(y) )=
\sum_{k=0}^\infty\, [\, (\F^{k}_{x})^{-1}(\p(f^{k}x))-  \H_{y,x}^s (\F^{k}_{y})^{-1}(\p(f^{k}y))\,]
$$
and satisfies $ \|\P_{x,y}^s\| \le K' d (x,y)^\beta$ with uniform $K'$ for all $x\in \M$ and $y \in W^s_{loc}(x)$.
\end{lemma}

The result holds if instead of being uniformly bounded $\F$ satisfies the following.
There exist numbers $\theta<1$ and $L$  such that for all $x\in\M$ and $n\in \N$,
$$
\|(\F_x^n)^{-1}\| \cdot  (\nu^n_x)^\beta < L\, \theta^n.
$$

\begin{proof}
For all $x\in \M$ and $y \in W^s_{loc}(x)$ we have $d(f^{k}x,f^{k}y)\le \nu_x^k \,d(x,y)$. As
 $\p$ is $\beta$-H\"older we obtain
 $$
 \|\p(f^{k}x)-  \p(f^{k}y)\|\le K_1(\nu^k_x d(x,y))^\beta,
 $$
 and since $\H_{f^ky,f^kx}^s$ is $\beta$-H\"older close to identity by $(\H 4)$, we have
 $$
 \|\p(f^{k}x)-  \H_{f^ky,f^kx}^s \p(f^{k}y)\|\le K_2(\nu^k_x\, d(x,y))^\beta.
 $$

 By uniformly boundedness of  $\F$ we have $\| (\F_x^k)^{-1} \| \le K$, and by continuity of $\phi$ we have $\sup_x \|\phi(x) \| \le K_3$.
Therefore,
$$
\P^n(x) - \H_{y,x}^s \P^n(y) =
\sum_{k=0}^{n-1}\,(\F^{k}_{x})^{-1}(\p(f^{k}x))-
(\H_{y,x}^s \circ (\F^{k}_{y})^{-1} \circ \H_{f^kx,f^ky}^s) (\H_{f^ky,f^kx}^s\p(f^{k}y))
$$
Since $\H_{y,x}^s \circ (\F^{k}_{y})^{-1} \circ \H_{f^kx,f^ky}^s= (\F^k_x)^{-1}$ by ($\H$3),\, the $k^{th}$ term in the sum equals
  $$
(\F^{k}_{x})^{-1}(\p(f^{k}x))-  (\F^{k}_{x})^{-1}  (\H_{f^ky,f^kx}^s \p(f^{k}y))=
(\F^{k}_{x})^{-1}\,[ \p(f^{k}x)-\H_{f^ky,f^kx}^s \p(f^{k}y)],
$$
and we estimate
   $$
   \begin{aligned}
\|(\F^{k}_{x})^{-1}\,&[  \p(f^{k}x)   -\H_{f^ky,f^kx}^s \p(f^{k}y)] \,\|  \le
 \|(\F^{k}_{x})^{-1} \| \cdot \|\p(f^{k}x)-\H_{f^ky,f^kx}^s \p(f^{k}y)  \| \le \\
 &\|(\F^{k}_{x})^{-1} \| \cdot  K_2(\nu^k_x\, d(x,y))^\beta\le  K K_2\,\theta^{k} d(x,y)^\beta \quad\text{for some }\theta<1.
 \end{aligned}
 $$
Hence the series converges and
$$
\| \P^n(x) - \H_{y,x}^s \P^n(y) \|  \le\, \sum_{k=0}^{n-1}K K_2 \, \theta^{k} d(x,y)^\beta \le K' d(x,y)^\beta,
$$
so the limit $\P_{x,y}^s$  satisfies $ \|\P_{x,y}^s\| \le K' d (x,y)^\beta$.
\end{proof}

%%%%

\begin{proposition} \label{twist-meas}
Let $\F$ be a  $\beta$-H\"older uniformly bounded cocycle over an Anosov diffeomorphism $f$
(or a hyperbolic system).
Let $\mu$ be an ergodic $f$-invariant  measure on $\M$ with full support and local product structure.

Let $\p  : \M \to \E$ be a $\beta$-H\"older section, and let $\q:\M\to \E$ be a $\mu$-measurable section satisfying \eqref{twisteq}.
Then $\q$ is  $\beta$-H\"older and
$$
\q(x)=\H_{y,x}^s \,\q(y)+\P_{x,y}^s \quad \text{for all $x\in X$ and  $y \in W^s(x)$}.
$$

\end{proposition}

\begin{proof}
Let $x\in \M$ and $y\in W^s(x)$.
Using equation \eqref{twisteq iter} for $\q(x)$ and $\q(y)$
we obtain
$$
\q(x)-\H_{y,x}^s \,\q(y)=
\P^n(x) - \H_{y,x}^s \P^n(y) + \Delta_n,
$$
where
$$
 \Delta_n= (\F_{f^{n-1}x})^{-1} (\q(f^{n}x)) -\H_{y,x}^s (\F_{f^{n-1}y})^{-1} (\q(f^{n}y)).
$$
By Lemma \ref{twist hol},  $(\P^n(x) - \H_{y,x}^s\P^n(y) )$ converges to $\P_{x,y}^s$.
\vskip.1cm

Now we show that $\|\Delta_n\|\to 0$ along a subsequence for all $x,y$ in a set of full measure.
First we note that
by property ($\H$3) we have
 $\H_{y,x}^s (\F_{f^{n-1}y})^{-1}=(\F_{f^{n-1}x})^{-1} \circ \H_{f^ny,f^nx}^s$. Hence
$$
\Delta_n=
(\F_{f^{n-1}x})^{-1} \left( \q(f^{n}x) -\H_{f^ny,f^nx}^s(\q(f^{n}y)) \right)= (\F_{f^{n-1}x})^{-1} ( \Delta_n'),
$$
where  $\Delta_n'=\q(f^{n}x) -\H_{f^ny,f^nx}^s(\q(f^{n}y))$. By uniform boundedness of $\F$  we obtain
$$
\| \Delta_n\|\le \|(\F_{f^{n-1}x})^{-1} \| \cdot \|\Delta_n'\|\le K  \|\Delta_n'\|.
$$
Since the section $\q:\M\to E$ is $\mu$-measurable, by Lusin's theorem there
exists a compact set $S\subset \M$ with $\mu(S)>1/2$ such that
$\q$ is uniformly continuous and hence bounded on  $S$. Let $Y$ be the set of points in $\M$
for which the frequency of visiting $S$ equals $\mu(S)$.
By Birkhoff Ergodic Theorem, $\mu(Y)=1$.
\vskip.1cm

If $x,y\in Y$, there exists a subsequence $n_i\to \infty$
such that  such that $f^{n_i}x, f^{n_i}y \in S$ for all $i$.
Since  $y \in W^s(x)$, $\,d(f^{n_i}x, f^{n_i}y)\to 0$
 and hence $\Delta_{n_i}' \to 0$ by uniform continuity and boundedness
of $\q$ on $S$ and property ($\H$4) of $\H^s$. Thus $\Delta_{n_i} \to 0$ and we obtain  that
$$
\q(x)=\H_{y,x}^s \,\q(y) +\P_{x,y}^s \quad\text{ for all $x,y\in Y$ with $y \in W^s(x).$}
$$
Since $\P_{x,y}^s$ is $\beta$-H\"older
on $ W^s_{\loc}(x)$ by Lemma \ref{twist hol}, we conclude that
$$
\|\q(x)- \H_{y,x}^s\,\q(y)\| \le K'd(x,y)^\beta\quad\text{ for all $x,y\in Y$ with $y \in W^s(x).$}
$$
Since $\H_{x,y}^s$ is $\beta$-H\"older by  property ($\H$4), this means that $\q$ is
essentially $\beta$-H\"older along $ W^s_{\loc}(x)$.

Similar arguments  for $y \in W^u_{loc}(x)$ show that $\q$ is also essentially $\beta$-H\"older along $ W^u_{loc}(x)$.
Hence $\q$ is $\beta$-H\"older by the local product structure of $\mu$ and of the stable and unstable manifolds.
 \end{proof}

%%%%%%%%%%%%%  Proof of Theorem 2.1   %%%%%%%%%%%%

\subsection{Proof of Theorem \ref{measurable conjugacy}} \label{proof of thm measurable}
For convenience, by taking inverse, we will work with a conjugacy $\c$ satisfying
\begin{equation} \label{Cback}
  \B_x=\c(fx)\,\A_x\,\c(x)^{-1} .
\end{equation}
First we observe that since $\lambda_+(\A,\mu)=\lambda_-(\A,\mu)$ and $\B$ is $\mu$-measurably conjugate to $\A$,  the following lemma implies that
$$
\lambda_+(\B,\mu)=\lambda_-(\B,\mu).
$$

\begin{lemma} \label{equal exp}
Let $\mu$ be an ergodic $f$-invariant measure.
If $\c$ is a $\mu$-measurable conjugacy between cocycles $\A$ and $\B$, then for $\mu$ a.e. $x$
and for each vector $0\ne u \in \E_x$ the forward (resp. backward) Lyapunov exponent of $u$
under $\A$ equals that of $\c_x(u)$ under $\B$.
\end{lemma}
\begin{proof}
We fix a set of positive measure $Y\subseteq \M$ such that for some $K$ we have
$\|\c_x\| \le K$ and $\|(\c_x)^{-1}\| \le K$ for all $x\in Y$. Then we choose an $f$-invariant set
of full measure $X \subseteq \M$ such that for every $x \in X$
\begin{itemize}
\item[(i)] the forward and
backward Lyapunov exponents under both $\A$ and $\B$ exist for each non-zero vector $v \in \E_x$, and
\item[(ii)] the frequency of visiting $Y$ under both forward and backward iterates of $f$ equals $\mu(Y)>0$.
\end{itemize}
For every  $x \in X$,\, $0\ne u \in \E_x$, and $n\in \Z$ we have
$$
n^{-1} \ln \| \B^n_x (\c_x(u)) \| = n^{-1} \ln \| \c_{f^nx} (\A^n_x (u)) \|.
$$
The limit of the left hand side as $n \to \infty$ (resp. $n \to -\infty$) is the forward (resp. backward) Lyapunov exponent of $\c_x(u)$ under $\B$. On the other hand, by the choice of $Y$,
the limit of the right hand side along a subsequence $n_i \to \infty$ (resp. $n_i \to -\infty$)
such that $f^{n_i}x \in Y$ equals the forward (resp. backward) Lyapunov exponent of $u$ under $\A$.
\end{proof}

We use the following results from \cite{KS13}. In the three theorems below, $f$ is a transitive
$C^{1+\text{H\"older}}$ Anosov diffeomorphism,  $\A,\B:E\to \E$ are $\beta$-H\"older linear cocycles over $f$,
and $\mu$ is an ergodic $f$-invariant measure with full support and local product structure.
%%U^\cE^

\begin{theorem} \cite[Theorem 3.9]{KS13}\label{reductionH}
Suppose that for every $f$-periodic point $p$ the invariant measure
$\mu_p$ on its orbit satisfies $\lambda_+(\A,\mu_p)=\lambda_-(\A,\mu_p)$.
 Then there  exist a flag of $\beta$-H\"older  $\A$-invariant sub-bundles
\begin{equation} \label{flagH}
\{0\} =U^0 \subset U^1 \subset ...  \subset U^{j-1} \subset U^k = \E
\end{equation}
and $\beta$-H\"older Riemannian metrics on the quotient bundles $U^{i}/U^{i-1}$,
$i=1, ... , k$, so that for some positive $\beta$-H\"older function
$\phi : \M \to \R$ the quotient-cocycles induced by the cocycle
$\phi \A$ on $U^{i}/U^{i-1}$ are isometries.

 \end{theorem}

\begin{theorem} \cite[Theorem 3.1 and Corollary 3.8]{KS13} \label{structure}
If $\B$ is fiber bunched, then any $\B$-invariant $\mu$-measurable
conformal structure on $\E$  coincides $\mu$-a.e. with a H\"older continuous
conformal structure.
\end{theorem}

 If a cocycle has more than one Lyapunov exponent,
then the corresponding Lyapunov sub-bundles are
invariant and measurable, but not continuous in general.
For a fiber bunched cocycle with only one Lyapunov
exponent, measurable invariant sub-bundles are continuous.

\begin{theorem}  \cite[Theorem 3.3 and Corollary 3.8]{KS13}\label{distribution}
Suppose that $\B$\, is fiber bunched and $\lambda_+(\B,\mu)=\lambda_-(\B,\mu)$.
Then any $\mu$-measurable $\B$-invariant sub-bundle of $\EE$  coincides $\mu$-a.e. with a H\"older continuous one.
\end{theorem}

We consider the flag $U^i$ for $\A$ given by Theorem \ref{reductionH}.
Denoting $\calU^i_x=\c(x) U^i_x$ we obtain  the corresponding
flag of measurable $\B$-invariant  sub-bundles
$$
\{0\} =\calU^0 \subset \calU^1\subset \calU^2 \subset \dots \subset \calU^k=\E.
$$
By Theorem \ref{distribution} we may assume that the sub-bundles $\calU^i$ are H\"older continuous.

The conformal structure $\sigma_1$ on $E^1$ given by the Riemannian metric in Theorem \ref{reductionH}
is invariant under $\A$. The push forward of $\sigma_1$  by $\c$ gives a measurable $\B$-invariant conformal structure $\tau_1$ on $\calU^1$, which is H\"older continuous by Theorem \ref{structure}.

 Similarly, we consider H\"older continuous quotient-bundles $\U^i=U^i/U^{i-1}$
and $\UU^i=\calU^i/\calU^{i-1}$ over $\M$ with the quotient cocycles $\A^{(i)}$ and $\B^{(i)}$.
Since $\A^{(i)}$  preserves a H\"older continuous conformal structure $\sigma_i$ on $\U^i$, pushing
forward by $\c$ we obtain a measurable conformal structure $\tau_i$ on $\calU^i/\calU^{i-1}$ invariant
under $\B^{(i)}$, which is H\"older continuous by Theorem \ref{structure}.
Thus we obtain a ``similar structure" for $\B$.

We fix a $\beta$-H\"older Riemannian metric on $\E$. We denote by $\V^i$  the orthogonal
complement of $U^{i-1}$ in $\E_i$, and we denote by  $\VV^i$ the orthogonal
complement of $\calU^{i-1}$ in $\calU^i$, $i=1,\dots,k$. Thus
$U^i=\V^1 \oplus \cdots \oplus \V^i$ and  $\calU^i=\VV^1 \oplus \cdots \oplus \VV^i$.
All these sub-bundles are H\"older continuous but for $i>1$ they are not invariant under $\A$ and $\B$,
and $\c$ does not necessarily map $\V^i$ to $\VV^i$.

 We denote by $P^{j}:\E \to \V^j$ the projection to the $\V^j$ component in the splitting
$\E=\V^1 \oplus \cdots \oplus \V^k$ and similarly $\PP^{j}:\EE \to \VV^j$.

We denote the restriction of $\c$ to $\V^i$ by $\c^i$ and we denote by $\c^{j,i}$ its $j$-component
 $\c^{j,i}=\PP^{j} \circ \c^i : \V^i \to \VV^j$.
 Since $\calU^i_x=\c(x) U^i_x$, we have $\c^i :\V^i \to \calU^i$ and thus $\c^{j,i}=0$ for $j>i$,
that is $\c$ has an upper triangular block structure.

We also define the corresponding blocks $\A^{j,i} : \V^i \to \V^j$ and $\B^{j,i} : \VV^i \to \VV^j$
as $\A^{j,i} =P^j\circ \A |_{ \V^i}$ and similarly for $\B$. The invariance of the flags also yields
 upper triangular block structures for $\A$ and $\B$: $\A^{j,i}=0=\B^{j,i}$ for $j>i$.

We will show inductively that the restriction of $\c$ to $U^i$ is H\"older continuous, $i=1,\dots,k$.
The base case $i=1$ follows from the following result from \cite{S15}.

\begin{theorem} \cite[Theorem 2.7]{S15} \label{QC}
Let $\A,\B:E\to \E$ be $\beta$-H\"older linear cocycles over a hyperbolic system.
Suppose that $\A$ uniformly quasiconformal  and $\B$ is fiber bunched.
Let  $\mu$ be an ergodic invariant  measure with full support and local product structure.
Then any $\mu$-measurable conjugacy between $\A$ and $\B$ is $\beta$-H\"older continuous,
i.e. it coincides with a $\beta$-H\"older continuous conjugacy on a set of full measure.
\end{theorem}

 Now we describe the inductive step. Assuming that the restriction of $\c$ to $U^{i-1}$ is $\beta$-H\"older
continuous we show that so is the restriction to $U^{i}$. Since $U^i=\V^i \oplus U^{i-1}$,
it suffices to show that the  restriction $\c^{i}$ of $\c$ to $\V^{i}$ is  also $\beta$-H\"older continuous.
We will establish this inductively for each of its components $\c^{j,i}$, $j=i,\dots ,1$.

First we observe that $\c^{i,i}$ is H\"older continuous for all  $i=1,\dots,k$. For this we identify bundles $\V^i$ with $\U^i$
and $\VV^i$ with $\UU^i$ via the projections. Under these identifications the cocycle
$\A^{i,i}:\V^{i,i} \to \V^{i,i}$ corresponds to the quotient cocycle $\A^{(i)}$, the cocycle
$\B^{i,i}:\VV^{i,i} \to \VV^{i,i}$ corresponds to  $\B^{(i)}$, and the map $\c^{i,i}$ corresponds to the
quotient measurable conjugacy $\c^{(i)}$ between  $\A^{(i)}$ and $\B^{(i)}$.
Since the quotient cocycles $\A^{(i)}$ and $\B^{(i)}$ are conformal, Theorem \ref{QC} shows that
 $\c^{(i)}$ is $\beta$-H\"older continuous, and hence so is $\c^{i,i}$.

Now we show that $\c^{i-\ell,i}$ is $\beta$-H\"older assuming that $\c^{i-j,i}$  is $\beta$-H\"older for $j=0,1,\dots \ell-1$.
Using the conjugacy equation
$$
\B_x \circ \c_x =\c_{fx}  \circ \A_x
$$
and equating $(i-\ell,i)$ components we obtain
$$
\begin{aligned}
&\B^{i-\ell,i-\ell}_x \circ \c^{i-\ell,i}_x + \B^{i-\ell,i-\ell+1}_x \circ \c^{i-\ell+1,i}_x + \dots +\B^{i-\ell,i}_x \circ \c^{i,i}_x \\
&
=\c^{i-\ell,i-\ell}_{fx}  \circ \A^{i-\ell+1,i}_x + \c^{i-\ell,i-\ell+1}_{fx}  \circ \A^{i-\ell+1,i}_x + \dots +\c^{i-\ell,i}_{fx}  \circ \A^{i,i}_x
\end{aligned}
$$
and hence
\begin{equation} \label{CD}
\c^{i-\ell,i}_x  = (\B^{i-\ell,i-\ell}_x)^{-1} \circ \c^{i-\ell,i}_{fx}  \circ \A^{i,i}_x \, + D_x
\end{equation}
where
$$
\begin{aligned}
D_x & = (\B^{i-\ell,i-\ell}_x)^{-1} \circ(\c^{i-\ell,i-\ell}_{fx}  \circ \A^{i-\ell+1,i}_x + \dots +\c^{i-\ell,i-1}_{fx} \circ \A^{i-1,i}_x )- \\
&-(\B^{i-\ell,i-\ell}_x)^{-1} \circ ( \B^{i-\ell,i-\ell+1}_x \circ \c^{i-\ell+1,i}_x + \dots +\B^{i-\ell,i}_x \circ \c^{i,i}_x ).
\end{aligned}
$$
We view $\c^{i-\ell,i}_x$ and $D_x$ as sections of the H\"older bundle $L(\V^i,\VV^{i-\ell})$
whose fiber at $x$ is the space of linear maps $L(\V^i_x,\VV^{i-\ell}_x)$.
Thus equation \eqref{CD} is of the form \eqref{twisteq}
with
$$
E=L(\V^i,\VV^{i-\ell}), \;\;\p _x =D_x,\;\; \eta_x=\c^{i-\ell,i}_x, \;\;\text{and}\;\; \F_x(\q_{fx})=  (\B^{i-\ell,i-\ell}_x)^{-1} \circ \eta_{fx}  \circ \A^{i,i}_x.
$$
We note that $D_x$ is $\beta$-H\"older since we inductively know that all its terms are $\beta$-H\"older.
Also $\F$ is a linear cocycle on the bundle $L(\V^i,\VV^{i-\ell})$ over $f^{-1}$, and it is  $\beta$-H\"older since so are
$\B^{i-\ell,i-\ell}$ and $ \A^{i,i}$. Moreover, $\F$ is uniformly bounded since cocycles $\B^{i-\ell,i-\ell}$
and $ \A^{i,i}$ are conformal and their normalizations are continuously cohomologous. The latter
follows since we know that $\B^{i-\ell,i-\ell}$ and $\A^{i-\ell,i-\ell}$  are continuously cohomologous
by $\c^{i-\ell,i-\ell}$ and that the normalizations of all $ \A^{i,i}$ are given by the same function $\phi ^{-1}$ from Theorem \ref{reductionH}.
Hence we can apply Proposition \ref{twist-meas} and conclude that $\c^{i-\ell,i}$ is $\beta$-H\"older.
\vskip.1cm
The argument above applies to $\ell=1, \dots i-1$ and we conclude that all $\c^{1,i}, \dots  ,\c^{i,i}$ are H\"older.
This proves that the restriction of $\c$ to $U^{i}$ is H\"older and completes the inductive step.
We conclude that $\c$ is H\"older, completing the proof of Theorem \ref{measurable conjugacy}.

%%%%%%%%%%  Proof of Theorem Constant Cocycle %%%%%%%%%%%%
%%%%%%%%%%%%%%%%%%%%%%%%%%%%%%%%%

 \section{Proof of Theorem \ref{constant cocycle}} \label{proof of constant}

In this proof we will also work with a conjugacy $\c$ satisfying \eqref{Cback}.
First, H\"older continuity of $\c$ is deduced from Theorem \ref{measurable conjugacy} as follows.

Let $A \in GL(N,\R)$ be the generator of the constant cocycle $\A$.
Let $\rho_1 < \dots <\rho_L$ be the distinct moduli of the eigenvalues of $A$ and let
\begin{equation} \label{splitA}
 \R^N = \E^1 \oplus  \dots \oplus \E^L
\end{equation}
be the corresponding invariant splitting as in \eqref{splitL}. In this section we will use the
adapted norm on $\R^N$ for which we have estimates  \eqref{rateA2}.
They imply that for any $\beta>0$ the cocycle $\A_i$ generated by $A_i$
is fiber bunched if $\e$ is sufficiently small.

 Let  $B(x)=\B_x:\M\to GL(N,\R)$ be the generator of the cocycle $\B$. If $B$ is
sufficiently $C^0$ close to $A$, then $\B$ has H\"older continuous invariant
splitting $C^0$ close to \eqref{splitA}
$$
\R^N = \EE^1_x \oplus \dots \oplus \EE^L_x,
$$
so that the restrictions $\B_i=\B| \EE^i$ satisfy  estimates similar to  \eqref{rateA2}
\begin{equation} \label{rateB}
 (\rho_i-2\e)^n \le \| \B_i^n u \| \le  (\rho_i+2\e)^n
\quad \text{for any unit vector }u\in \EE^i.
\end{equation}
This is well known but also follows from Lemma \ref{Ei est}, which gives explicit estimates
of both H\"older exponent and H\"older constant.
We conclude that all restrictions $\B_i$ are $\beta$-H\"older and hence are fiber bunched
 if $\e$ is sufficiently small.

Let $\c$ be a measurable conjugacy between $\A$ and $\B$.
We claim that  $\c$ maps $\E^i$ to $\EE^i$, that is
$\c_x (\E^i)=\EE^i_x$ for $\mu$ a.e. $x$. Indeed, by Lemma \ref{equal exp}, for $\mu$ a.e. $x$ and for each unit vector $u \in \E^i$
the forward and backward Lyapunov exponent of $\c_x(u)$ is $\ln \rho_i$.
This yields that $\c_x(u) \in \EE^i$, as having a non-zero component in another $\EE^j$ would
imply having forward or backward Lyapunov exponent under $\B$ different from $\ln \rho_i$
if $\e$ is sufficiently small.
Then $\c_i=\c|_{ \E^i}$ is a measurable conjugacy between fiber bunched cocycles $\A_i$ and $\B_i$.
By Theorem \ref{measurable conjugacy} each $\c_i$ is H\"older for all $i=1,\dots,L$,
 and hence so is $\c$.

\vskip.3cm

Now we prove the more detailed statement. We denote the Lipschitz constants of $f^{-1}$ and $f$ respectively by
\begin{equation} \label{alpha}
\a_f= \sup_{x \in \M} \| D_xf^{-1}\| >1 \quad\text{and}\quad \a_f'=\sup_{x\in \M} \| D_xf\| >1.
\end{equation}
For $1\le i<L$ we define
$$
\beta_i= \text{\small{$\frac{\ln (\rho_{i+1}/\rho_{i})}{ \ln (\a_f)}$}} \quad\text{and}\quad
\beta_i'= \text{\small{$\frac{\ln (\rho_{i+1}/\rho_{i})}{\ln (\a_f')}$}},
$$
and we choose
\begin{equation} \label{beta 0}
\beta_0= \beta_0(A,f) =\min\,\{1,\, \beta_1, \dots, \beta_{L-1},\, \beta_1', \dots, \beta_{L-1}'\} >0.
\end{equation}

Since $\B$ is $\beta$-H\"older with $\beta \le \beta' <\beta_0$, Lemma \ref{Ei est} below shows that the splitting \eqref{rateB}  is $\beta$-H\"older
and by Lemma \ref{Bi est} so are all restrictions $\B_i$. Then by Theorem \ref{measurable conjugacy} each restriction
$\c_i=\c|_{\E^i}$  is $\beta$-H\"older and hence so is $\c$.
Since $\A_i$ and $ \B_i$ are $\beta$-fiber bunched for any sufficiently small $\e$,
  \cite[Proposition 4.5]{S15} yields that $\beta$-H\"older $\c_i$ intertwines their stable holonomies, that is,
\begin{equation}\label{intertwines2}
\H_{x,y}^{\A_i,s}=\c_i(y)\circ \H_{x,y}^{\B_i,s}\circ \c_i(x)^{-1}\quad \text{for all }x,y \in \M
\text{ such that }y\in W^s(x).
\end{equation}
Since for the constant cocycle $\A_i$ the holonomies  are all identity, $\H_{x,y}^{\A_i,s}=\Id$, we get
$$\c_i(x)=\c_i(y)\circ \H_{x,y}^{\B_i,s}.
$$
 Thus using Lemma \ref{Hi est}  we obtain that for all $y\in W^s(x)$
$$\| \c_i(x)- \c_i(y)\|=\|\c_i(y) \circ (\H_{x,y}^{\B_i,s}-\Id)\| \le  \| \c_i\|_{C^0}  \cdot k_3 \, \hk_\beta (\B) \cdot  d_{W^s}  (x,y)^{\beta}.
$$
Combining these estimates for all $i=1,\dots, L$ we conclude that all $y\in W^s(x)$
$$\| \c(x)- \c(y)\| \le  \| \c\|_{C^0}  \cdot k_4 \, \hk_\beta (\B) \cdot  d_{W^s}  (x,y)^{\beta}.
$$
Similarly, using the analog of Lemma \ref{Hi est} for unstable holonomies, we obtain
 the same estimate for $y\in W^u(y)$. Then the local product structure of stable and unstable
  foliations of $f$ implies that the $\beta$-H\"older constant of $\c$
can be estimated as
$$
 \hk_\beta(\c) \le k \, \| \c\|_{C^0} \,  \hk_\beta (\B).
$$
Now, to complete the proof of the second part of the theorem,
we state and prove the lemmas used in the above argument.

\vskip .2cm

%%%%%%%%%%%%%%%%%%%%%%%%%%

\begin{lemma}\label{Ei est}
For any $0<\beta'<\beta_0$ there is $\delta >0$ and $k_1>0$ such that for any $0<\beta \le \beta'$
 any  $\beta$-H\"older $GL(N,\R)$ cocycle $\B$ with $ \|\B_x-A\|_{C^{0}}<\delta$ preserves $\beta$-H\"older splitting
$$
\R^N = \EE^1_x \oplus \dots \oplus \EE^L_x
$$
 which is $C^0$ close to $\E^1 \oplus  \dots \oplus \E^L$ and for each $1\le i\le L$ the
 $\beta$-H\"older constant $\hk_\beta(\EE^i)$ of $\EE^i$ satisfies
\begin{equation} \label {hkEi}
 \hk_\beta(\EE^i) \le k_1\,  \hk_\beta (\B)%= k_1 \hk_\beta (\B_x -A).% \|\B_x-A\|_{C^{\b}}.
 \end{equation}
\end{lemma}

\begin{proof} We deduce this lemma from the one below. We fix $1\le i<L$, and let
$$
E'= \E^1 \oplus  \dots \oplus \E^i\quad\text{and}\quad  E=\E^{i+1} \oplus  \dots \oplus \E^L.
$$
Lemma \ref{EE' est} below shows that for any $\beta'<\beta_i$ there is $\delta >0$ and
$k'$ such that for any $0<\beta \le \beta'$  any cocycle $\B$ with $ \|\B_x-A\|_{C^{0}}<\delta$ preserves
the bundle $\EE$ close to $\E$ with the desired estimate for  $\beta$-H\"older constant.
Similarly, for any $\beta'<\beta_i'$ using the inverses of $A$ and $f$ we obtain that $\B$
preserves a bundle $\EE'$ close to $\E'$ with a similar estimate for its $\beta$-H\"older constant.
Then  for each $1\le i\le L$ the bundle $\EE^i$ is defined as a suitable intersection and hence
is also $C^0$ close to $\E^i$ and its  $\beta$-H\"older constant satisfies \eqref{hkEi}.
\end{proof}

\begin{remark}
Lemmas  \ref{Ei est} and \ref{EE' est} do not rely on hyperbolicity of $f$ and use only that it is bi-Lipschitz.
\end{remark}
%%%%%%%%%%\xi'

\begin{lemma}\label{EE' est}
Let $A \in GL(N,\R)$, let $\,\R^N=E'\oplus E$ be an $A$-invariant splitting, and let
   $$
   \begin{aligned}
   \xi' &= \max \,\{\,\| Av \| : \, v \in E',\;  \|v\|=1\,\}=\|A|_{E'}\|\,\,\text{ and } \\
    \xi & = \min \,\{\,\|Av \| : \, v \in E,\;  \|v\|=1\,\} = \|A^{-1}|_{E}\|^{-1}.
   \end{aligned}
   $$
Let $\a_f=\sup \| Df^{-1}\| >1$ be the Lipschitz constant of $f^{-1}$ and let $\beta'>0$. Suppose that
$$\,
\xi'<\xi \quad\text{and}\;\quad \text{\small{$\frac{\xi' \a_f^{\beta'}}{\xi}$}} <1,
\;\text{  that is, }\;\,\beta'< \text{\small{$\frac{\ln (\xi /\xi')}{\ln\a_f}$}}.
$$
Then there is $\delta >0$ and $k'$ such that for any $0<\beta \le \beta'$
 any $\beta$-H\"older $GL(N,\R)$ cocycle $\B$ with $ \|\B_x-A\|_{C^{0}}<\delta$ preserves
a $\beta$-H\"older sub-bundle $\EE$ which is $C^0$ close to $E$ and its $\beta$-H\"older constant $\hk_\beta(\EE)$ satisfies
$$\hk_\beta(\EE) \le k'\,  \hk_\beta (\B)$$ % \|\B_x-A\|_{C^{\b}}.$$
\end{lemma}

\begin{proof}
The argument is similar to the H\"older version the $C^r$ Section Theorem of
M.~Hirsch, C.~Pugh, and M.~Shub (see Theorem 3.8 in  \cite{HPS}), but we give
the  estimate  of the H\"older constant.

We consider the space $\L=\L(E,E')$ of linear operators from $E$ to  $E'$ and endow it with the
standard operator norm.  Since $A$ preserves the splitting $E'\oplus E$ it induces the graph transform
action $\hat A$ on $\L$ as follows: if $L \in \L$ and $G \subset \R^N$  is its graph then $\hat A (L) $ is
 the operator in $\L$ whose graph is $A (G)$. The map $\hat A$ is linear,
 $$
 \hat A \,[L] = A |_{E'} \circ L \circ (A |_{E})^{-1},
 $$
so we can estimate its norm as
$$
\|\hat A \| \le \| A |_{E'}\| \cdot \|(A |_{E})^{-1}\|\le \xi'/\xi <1.
$$

Similarly, any linear map $B\in GL(N,\R)$ sufficiently close to $A$ induces in the same way
the graph transform map $\hat B$ on a unit ball $\L_1$ in $\L$. Moreover, $\hat B$ is a contraction
 of $\L_1$ with Lipschitz constant $K(\hat B)$ close to $K(\hat A )=\xi'/\xi<1$. Indeed, $B$ induces
 an algebraic map on the Grassmannian of $(\dim E)$-dimensional subspaces which, together
 with its first derivatives, depends continuously on $B$.  Also, it is easy to see that the map
 $B \mapsto \hat B$ from a small neighborhood of $A$ to $C^0(\L_1,\L_1)$ is Lipschitz with some
   constant $\hat L$.

Now we consider the trivial fiber bundle $\vv=\M \times \L_1$. Then any $\B_x$ which is  $C^0$-close to $A$
 induces graph transform maps $\hat \B_x : \vv_x \to \vv_{fx}$ and thus
the bundle map $\hat \B : \vv \to \vv$ covering $f$. We consider the space $S$ of continuous
sections of $\vv$ with the supremum norm, and the induced action $F=F_\B$ on $S$
defined for $s \in S$ as $(Fs)(fx)= \B_x(s(x))$. If $K^{\B}:=\sup _x K({\hat \B_x }) < 1$ then $F$ is
a contraction on $S$ and hence has a unique fixed point $s_*=Fs_*$. Let $s_0(x)=0\in \L$ be the
zero section, then we can write $s_*= \lim F^ns_0$ and it follows that $s_*$ is $C^0$-close to $s_0$.
Denoting the graph of $s(x)$ by $\EE_x$ we obtain the unique continuous $\B$-invariant sub-bundle close to $E$.

Now we will show that $s_*$ is  $\beta$-H\"older and estimate its $\beta$-H\"older constant.
For this we will find $M>0$ such that $\hk_\beta (s) \le M$ implies $\hk_\beta (F s) \le M$.
Then $\hk_\beta (F^n (s_0)) \le M$ for all $n$ and since $s_*= \lim F^n(s_0)$ it will follow that
$\hk_\beta (s_*) \le M$.

Fix points $z,z'$ and  let $x=f(z)$, $x'= f(z')$. Then for any $\beta$-H\"older $s\in S$ we can
estimate,  as $\|s(x)\| \le 1$, that
$$
\begin{aligned}
&\| Fs (x) - Fs(x')\|=\| \hat \B_z s (z) - \hat \B_{z'}s(z')\|  \\
& \le \| \hat \B_z s (z) - \hat \B_{z'}s(z)\| + \| \hat \B_{z'} s (z) - \hat \B_{z'}s(z')\| \\
&\le d_{C^0}(\hat \B_{z} ,\hat \B_{z'} )  + K({\hat \B_{z'}  })\|  s (z) - s(z')\| \le \hat L \| \B_{z} -\B_{z'}\|  + K^{\B}\|  s (z) - s(z')\| \\
&\le \hat L\,  \hk_\beta (\B) \,d(z,z')^\beta  + K^{\B} \hk_\beta (s) \,d(z,z')^\beta \le
 [\hat L \, \hk_\beta (\B)  + K^{\B} \hk_\beta (s)]\,(\a_f\, d(x,x'))^\beta,
 \end{aligned}
$$
where $\a_f$ is the Lipschitz constant of $f^{-1}$ and $\hat L$ is the Lipschitz constant of the map $B \mapsto \hat B$ on a neighborhood of $A$. Hence $Fs$ is also $\beta$-H\"older and
$$
\hk_\beta (Fs) \le \hat L \, \a_f^\beta\, \hk_\beta (\B)  +  \a_f^\beta \,K^{\B} \hk_\beta (s).
$$
Therefore, $\hk_\beta (s) \le M$ implies $\hk_\beta (Fs) \le M$ if we take
$$
M= (1-K^{\B}a_f^\beta)^{-1} \hat L \,a_f^\beta\,\hk_\beta (\B).
$$
If $\|\B_x-A\|_{C^0}$ is small then  $K^{\B}$ is close to $K(\hat A )=\xi'/\xi$.
 Since $\,\xi' \a^{\beta'} /\xi <1\,$ and $\beta \le \beta'$ it follows that $1-K^{\B}a_f^\beta>0$
 and is separated from $0$.
Then there is a constant $k'$ which bounds $ \hat L a_f^\beta\,(1-K^{\B}a_f^\beta)^{-1}$ for
all  $0<\beta \le \beta'$ and all $\B$ with $\|\B_x-A\|_{C^{0}}<\delta$. Hence,
%using $\hk_\beta (\B) \le \|\B_x-A\|_{C^{\b}}$ we conlcude that
 $$
 M \le k' \, \hk_\beta (\B). %\le k'\cdot  \|\B_x-A\|_{C^{\b}}
 $$
\vskip.1cm

Finally, since $\hk_\beta (s_0) =0$  it  follows that $\hk_\beta (F^n (s_0)) \le M$ for all $n$ and hence for the limit
 we also have $\hk_\beta (s_*) \le M \le k'  \hk_\beta (\B) $. %\cdot  \|\B_x-A\|_{C^{\b}}$.
\end{proof}
%%%%%%%%%%%%

\vskip.2cm

Now we estimate  the $\beta$-H\"older constants of the restricted cocycles $\B_i=\B|_{ \EE^i}$.

\begin{lemma}\label{Bi est} For any $0<\beta'<\beta_0$ there is $\delta >0$ and $k_2>0$ such that
for any $0<\beta \le \beta'$ and any $\beta$-H\"older cocycle $\B$ with $\|\B_x-A\|_{C^{0}}<\delta$ the $\beta$-H\"older
constant of the cocycle $\B_i$, $i=1,\dots L$, satisfies
$$
\hk_\beta(\B_i) \le k_2\, \hk_\beta (\B).% \|\B_x-A\|_{C^{\b}}.
$$
\end{lemma}

\begin{proof}
Denoting $B(x)=\B_x$ and $B_i(x)=\B_x| _{\EE^i}$ we need to estimate the distance between $B_i(x)$ and $B_i(y)$.
To do this using their difference, we fix $\beta$-H\"older identifications $\I_{x,y} :\EE^i_x \to \EE^i_y$,
say by translation from $x$ to $y$ in the trivial bundle $\M \times \R^N$ followed by an appropriate rotation.
Then for a unit vector $u\in \EE^i(x)$ we need to estimate $\|(B_i(x)-B_i(y)\circ \I_{x,y})u\|$.
We note that
$$
\|u-\I_{x,y}u\| \le \dist (\EE^i_x,\EE^i_y)\le \hk_\beta(\EE^i) \,d(x,y)^\beta.
$$
Also, since $B(x)$ is $\beta$-H\"older have
$\|B(x)u-B(y)u\| \le \hk_\beta(\B) \,d(x,y)^\beta$. % \le  \|\B_x-A\|_{C^{\b}} \cdot d(x,y)^\beta.
 Hence we obtain that for a unit vector $u\in \EE^i(x)$
$$
\begin{aligned}
\|(B_i(x)-B_i(y)\circ \I_{x,y})u\|& \le  \|B(x)u-B(y)u\|+\|B(y)\| \cdot \|u-\I_{x,y}u\| \\
& \le \hk_\beta (\B) \, d(x,y)^\beta+  \|B\|_{C^0} \,\hk_\beta(\EE^i) \, d(x,y)^\beta .
\end{aligned}
$$
Since
 $\hk_\beta(\EE^i) \le k_1\, \hk_\beta (\B)$ by \eqref{hkEi} and
 $\|B\|_{C^0} \le \|A\| +\|\B_x-A\|_{C^{0}}\le \|A\| +\delta$  we conclude that
$$
\|(B_i(x)-B_i(y)\circ \I_{x,y})u\| \le  k_2 \, \hk_\beta (\B) \, d(x,y)^\beta.
$$
Thus $\hk_\beta(\B_i) \le k_2\,  \hk_\beta (\B) .$% \|\B_x-A\|_{C^{\b}}.$
\end{proof}

%%%%%%%%%%%%%%%%%%%%

In the next lemma we consider the stable holonomies of cocycles $\B_i=\B|_{\EE^i}$, $i=1,\dots, L$.

\begin{lemma}\label{Hi est}
For any $0<\beta'<\beta_0$ there is $\delta >0$ and $k_3>0$ such that for any $0<\beta \le \beta'$
and a $\beta$-H\"older cocycle $\B$ with $\|\B_x-A\|_{C^{0}}<\delta$ the holonomies of cocycles  $\B_i=\B|_{\EE^i}$ satisfy
$$
\| \H^{s}_{\,x,y} - \Id \,\| \le k_3 \, \hk_\beta (\B) \, d (x,y)^{\beta} \;\text{ for any  } x\in \M \text{ and } y\in W^s_{\text{loc}}(x).
$$

\end{lemma}

 \begin{proof}

%Existence of holonomies under the weaker fiber bunching  of Definition \ref{bunching def} can be deduced by considering iterates of the cocycle and using uniqueness of holonomies, see \cite[Proposition 4.4]{S15} for details.d
We fix $i$ and denote $\F=\B_i$. The stable holonomies of $\F$ are given by
\begin{equation}\label{hol def F}
 \H^{\F,s}_{x,y} =\underset{n\to\infty}{\lim} \,(\F^n_y)^{-1} \circ \F^n_x.
\end{equation}
The existence is ensured by fiber bunching of $\F$. Indeed,  the contraction along $W^s$
is estimated by \eqref{Anosov def} as
$$
d (f^nx, f^ny)\le  \nu^n d (x,y)
\;\text{ for any  } x\in \M, \;\, y\in W^s_{\text{loc}}(x),\;\, n\in \N,
$$
We also obtain from  \eqref{rateB}  that
  \begin{equation}\label{F norm}
\| \F_x^m\|\cdot \|(\F_y^m)^{-1}\| \le  \prod_{j=0}^{m-1} \|\F_{x_j}\| \, \|(\F_{y_j})^{-1}\| \le \,
\text{\small {$\left( \frac{\rho_i+2\e}{\rho_i-2\e} \right)$}}^m =\sigma^m \quad\text{for all }x,y\in\M,
\end{equation}
where $\sigma= (\rho_i+2\e) (\rho_i-2\e)^{-1}$ is close to $1$ when $\e$ is small. % if $\delta$ and hence $\e$ are sufficiently small.
 It follows that
 \begin{equation}\label{F bunch}
\| \F_x^m\|\cdot \|(\F_y^m)^{-1}\|  \cdot  \nu^{m\beta} \le \sigma^m  \cdot  \nu^{m\beta} =\theta^m \quad\text{for all }x,y\in\M,
\end{equation}
where $\,\theta=\sigma  \nu^{\beta}<1\,$ if $\delta$ and hence $\e$ are sufficiently small. In particular, $\F$ is fiber bunched so the limit in \eqref{hol def F} exits, though this also follows from the proof.
\vskip.1cm

We want to obtain a constant $c$ such that $\| \H^{\F,s}_{\,x,y} - \Id \,\| \leq c\,d (x,y)^{\beta}\,$  for all $x\in \M$   and $y\in W^s_{\text{loc}}(x)$.
% and by \eqref {hol def F} it suffices to get
%$$\|(\F^n_y)^{-1} \circ \F^n_x - \Id \,\| \leq c\,d (x,y)^{\beta}\,\quad\text{for every $n\in \N$.}$$
%Let $x\in X$ and $y\in W^s_{\text{loc}}(x)$.
%The key step is to show that the sequence  $((\F^n_y)^{-1}\circ \F^n_x )$ is Cauchy.
Denoting $\,x_m=f^m(x)$ and $\,y_m=f^m(y)$, we obtain
$$
 \begin{aligned}
 & (\F^n_y)^{-1}\circ \F^n_x \,= (\F^{n-1}_y)^{-1}\circ \left(
   (\F_{y_{n-1}})^{-1} \circ \F_{x_{n-1}}\right) \circ \F^{n-1}_x \\
  & = (\F^{n-1}_y)^{-1} \circ (\Id+r_{n-1}) \circ \F^{n-1}_x  = (\F^{n-1}_y)^{-1}\circ\F^{n-1}_x+(\F^{n-1}_y)^{-1}\circ r_{n-1}
  \circ \F^{n-1}_x   \\
 &=\dots = \Id+\sum_{m=0}^{n-1} (\F^{m}_y)^{-1}\circ r_{m}\circ \F^m_x ,
 \quad \text{where  }r_m=(\F_{y_m})^{-1} \circ \F_{x_m}-\Id.
 \end{aligned}
$$
Since $\F$ is $\b$-H\"older, denoting $c'=   (\rho_i-2\e)^{-1}\hk_\beta (\F)$, we obtain that for every $m\ge 0$
$$
  \|r_m\|  \,\le\,  \|(\F_{y_m})^{-1}\| \cdot \| \F_{x_m} - \F_{y_m}\|
   % \le  c_1\,d(x_i, y_i)^\beta
    \le \|\F^{-1}\|_{C^0} \, \hk_\beta (\F)\, d(x_m, y_m)^\beta  \le c'\,d(x,y)^\beta \nu^{m\beta}.%d(x, y) \,\nu^i
$$
Using  \eqref{F bunch} it follows that
$$
\begin{aligned}
& \|(\F^{m}_y)^{-1}\circ r_{m}\circ \F^m_x\| \le
   \|(\F^{m}_y)^{-1}\| \cdot \|\F^m_x\|  \cdot c'\,d(x,y)^\beta \nu^{m\beta} \le \theta ^m \,c'\,d(x,y)^\beta .
 \end{aligned}
$$
 Therefore, for every $n\in \N$,
$$
  \|\Id-(\F^n_y)^{-1}\circ \F^n_x\|  \,\le\, \sum_{i=0}^{n-1}
  \| (\F^{i}_y)^{-1}\circ r_{i}\circ \F^i_x \| \le
c'\, d(x,y)^\beta \, \sum_{i=0}^{n-1}
 \theta^i \le c \,d(x,y)^\beta,
  $$
  where
   $$
 c= \text{\small {$\frac{c'}{1-\theta} $}} \,\le\,
 \text{\small {$ \frac { (\rho_i-2\e)^{-1} \hk_\beta (\F)}{1-\sigma  \nu^{\beta}}$}}  =k_3' \,  \hk_\beta (\F)
 \quad\text{with}\quad k_3'= (\rho_i-2\e)^{-1} (1-\sigma  \nu^{\beta})^{-1}.
   $$
   By  \eqref{hol def F} the sequence $\{(\F^n_y)^{-1}\circ \F^n_x\}$ converges to $\H^{\F,s}_{xy}$
 (in fact the  estimates imply that it is Cauchy) and the limit satisfies
 $$
\| \H^{s}_{\,x,y} - \Id \,\| \le c \, d (x,y)^{\beta} \;\text{ for any  } x\in \M \text{ and } y\in W^s_{\text{loc}}(x).
$$
By Lemma \ref{Bi est} we have
$
\hk_\beta (\F) =\hk_\beta(\B_i) \le k_2 \,  \hk_\beta (\B)
$
and we conclude that
 $$
\| \H^{s}_{\,x,y} - \Id \,\| \le k_3 \,\hk_\beta (\B) \, d (x,y)^{\beta} \;\text{ for any  } x\in \M \text{ and } y\in W^s_{\text{loc}}(x).
$$
This completes the proof of Lemma \ref{Hi est}
\end{proof}
%for  $k_3= k_3' k_2$.

%%%%%%%%%%%%%%%%%%%%%%%%%%%%%%%%%%%%%%%%%
%%%%%%%%%%%%%%%%%%%%%%%%%%%%%%%%%%%%%%%%%
%%%%%%%%%%%%%%% End Cocycle Part %%%%%%%%%%%%%%%%%%%
%%%%%%%%%%%%%%%%%%%%%%%%%%%%%%%%%%%%%%%%%
%%%%%%%%%%%%%%%%%%%%%%%%%%%%%%%%%%%%%%%%%

%newpage
\section{Proof of Theorem \ref{HolderConjugacy}} \label{proof HC}
Any two continuous conjugacies between $f$ and $A$ differ by an
element of the centralizer of $A$. By \cite[Corollary 1]{W}, any homeomorphism commuting with an
ergodic, in particular hyperbolic, automorphism $A$ is an affine automorphism, and hence all conjugacies
have the same regularity.

First, using Theorem~\ref{constant cocycle} we will show in Section \ref{C1H} that $H$ is a
$C^{1+\text{H\"older}}$ diffeomorphism, and moreover the H\"older constant of its derivative
satisfies the estimate
\begin{equation} \label{hkDH}
\hk_\beta(DH) \le k \, \| DH\|_{C^0} \, \|{f-A}\|_{C^{1+\beta}}.
\end{equation}
This part  does not rely on closeness of $H$ to the identity and the estimate applies to any conjugacy $H$.
Then  in Section \ref {PGE} we use \eqref{hkDH} and an interpolating inequality to obtain the desired estimate \eqref{C1H est} of $\|{H-I}\|_{C^{1+\beta}}$ for the conjugacy $C^0$ close to the identity.

%%%%%%%%%%%%%%%%%%%%%%%%

\subsection{Proving that $H$ is a $C^{1+\text{H\"older}}$ diffeomorphism} \label{C1H} $\;$\\
First we recall some properties of a map $g \in W^{1,q}(\R^N,\R^N)$, $q>N$, which also extend
to the case when $g \in W^{1,q}(\T^N,\T^N)$. It is well known that, as a consequence of Morrey's
inequality, for any such $g$ the Jacoby matrix of weak partial derivatives gives the differential $D_x g$ %(in the strong, Frechet sense)
 for almost every $x$ with respect to the Lebesgue measure $\mu$. Also, any such $g$ satisfies
 {\it Lusin's N-property}\, \cite{MM} %also [8, Theorem C] p=N + Holder
that $\mu (E)=0$ implies $\mu (g(E))=0$, as well as
 {\it Morse-Sard property}\,  \cite{P} that  $\mu (g(\EuScript{C}_g))=0$ for the set of critical points of $g$
$$
\EuScript{C}_g=\{x\in \T^N:\; D_x g \text{ exists but is not invertible}\},
$$
see also \cite{KK} for sharper results and further references.

\vskip.1cm

Now we assume that $H\in W^{1,q}$ with $q>N$, so that the differential $D_x H$ exists $\mu$-a.e., and
for the set
$$
G_{ H}=\{x\in \T^N:\; D_x   H \text{ exists}\}
\;\text{ and its complement }\;E_{H}=\T^N \setminus G_{ H}
$$
we have $\mu(G_H)$=1 and $\mu(E_H)=0$.
Further $G_{ H}=\EuScript{C}_{ H}\cup R_{ H}$ is the disjoint union of two measurable sets,
the critical set $\EuScript{C}_{ H}$ and the regular set
$$
R_{ H}=\{x\in \T^N: \;D_x  H \text{ is invertible}\}.
$$
Since $f$ and $A$ are diffeomorphisms, it follows from the conjugacy equation  $H\circ f=A\circ H$ that the sets $G_{ H}$, $\EuScript{C}_{H}$, and $R_{ H}$ are $f$-invariant. Further, differentiating the equation on the set $G_H$ we obtain
\begin{equation} \label{DH2}
D_{fx} H \circ D_x f=A \circ D_x  H.
\end{equation}
Denoting $\c(x)= D_x H$ on the set $R_H$ we obtain the conjugacy equation over $f$
\begin{equation} \label{ABconj2}
A=\c (fx)  \circ  \B_{ x} \circ \c(x)^{-1} \quad \text{for cocycles $\; \B_x=D_{x} f$ and  $\A_x=A.$}
\end{equation}

Now we show that $ \mu (R_H)=1$ and also that $f$ preserves a measure $\tilde \mu$ equivalent to $\mu$.
Since $\mu (E_H)=0$, the Lusin's N-property of $H$ yields $\mu ( H(E_H))=0$. Also, we have
 $\mu ( H(\EuScript{C}_H))=0$ by the Morse-Sard property. Hence for $R'_H=H(R_H)$ we
 have $\mu(R'_H)=1$.
 %Since $R_H$ is $f$-invariant, $R'_H$ is $A$-invariant and ergodicity of $A$ yields that $\mu(R'_H)=1$.
% This also gives $\mu(R_H)>0$, as otherwise $\mu(R'_H)=0$ by the  Lusin's N-property.
Now we consider the measure $\tilde \mu=(H^{-1})_*(\mu)$ and note that $\tilde \mu(R_H)=1$ as $\mu(R'_H)=1$.
Since $H$ is a topological conjugacy between $f$ and $A$, the measure $\tilde \mu$ is $f$-invariant and, in fact,
is the Bowen-Margulis measure of maximal entropy for $f$, since $\mu$ is that for $A$. Indeed, denoting the topological entropy by $\mathbf{h}_{top}$ and metric entropy with respect to $\tilde \mu$ by
$\mathbf{h}_{\tilde \mu}$ we get
$$
 \mathbf{h}_{\tilde \mu}(f)= \mathbf{h}_{\mu}(A) = \mathbf{h}_{top}(A) =\mathbf{h}_{top}(f).
$$
 In particular, $\tilde \mu$ is ergodic with full support and local product structure.
 Since $\c$ is a conjugacy between $\B$ and $A$  on $R_H$ with $\tilde \mu(R_H)=1$, by Lemma \ref{equal exp} we obtain that the Lyapunov exponents $\lambda_i^{f,\tilde \mu}$ of $\tilde \mu$ for the cocycle  $\B=D f$ are equal to the Lyapunov exponents $\lambda_i^{A}$ of $A$. Hence
 the sum of positive Lyapunov exponents (counted with multiplicities) for $\tilde \mu$  equals its entropy
$$
 \mathbf{h}_{\tilde \mu}(f)= \mathbf{h}_{\mu}(A)= \sum_{\lambda_i^{A} >0} \lambda_i ^A
 =\sum_{\lambda _i^{f,\tilde \mu} >0}\lambda _i^{f,\tilde \mu}.
$$
Thus we have equality in the Pesin-Ruelle formula, which implies that $\tilde \mu$  has absolutely continuous conditional measures on the unstable foliation of $f$ \cite{L}. Similarly, equality of the negative Lyapunov exponents yields that $\tilde \mu$ has absolutely continuous conditional measures on the
stable foliation of $f$. We conclude that $\tilde \mu$ itself is absolutely continuous. Moreover,
the density $\sigma (x)=\frac{d \tilde \mu}{ d \mu}$ is smooth and positive as a measurable solution
of the coboundary  equation $\sigma (fx)\sigma (x)^{-1}=\det Df(x)$. Thus $\tilde \mu$ is
equivalent to $\mu$, so that $\tilde \mu(R_H)=1$ implies $ \mu(R_H)=1$.

Provided that $\|A-\B_x\|_{C^0}=\|A-D_xf\|_{C^0} \le \|A-f\|_{C^1}< \delta$, where $\delta>0$ is from Theorem~\ref{constant cocycle}, we can apply this theorem with $f$ and $\tilde \mu$  to obtain that
$$\c(x)= D_x  H :\T^N \to GL(N,\R)$$
 coincides with a H\"older continuous function almost everywhere
with respect to $\tilde \mu$ and hence $\mu$. Since $H\in W^{1,q}$ we conclude that
$H$ is $C^{1+\text{H\"older}}$. Also, since $(D_x  H)^{-1}= \c(x)^{-1}$ exists and is also H\"older
continuous we see that $H$ is $C^{1+\text{H\"older}}$ diffeomorphism.
Further, Theorem~\ref{constant cocycle} gives us the estimate \eqref{hkDH},
which we will use to obtain the desired estimate for $\|{H-\Id}\|_{C^{1+\beta}}$ in Section \ref{PGE}.
This completes the proof that $H$ is $C^{1+\text{H\"older}}$ diffeomorphism assuming that $H\in W^{1,q}$.

\vskip.15cm

Now we consider the case when $\tilde H = H^{-1}$ is in $W^{1,q}$ and hence
 $D_x \tilde H$ exists $\mu$-a.e. We similarly define the sets $G_{\tilde H}$, $E_{\tilde H}$, $\EuScript{C}_{\tilde H}$, and $R_{\tilde H}$, which are measurable and $A$-invariant. Hence by ergodicity of $A$
 the set $R_{\tilde H}$ must be null or co-null for $\mu$.
If $\mu (R_{\tilde H})=0$  then $\mu ( \tilde H(R_{\tilde H}))=0$ by the Lusin's N-property of $\tilde H$,
 but this  is impossible since  $\mu (\tilde H(E_{\tilde H}))=0$ by the Lusin's N-property  and
  $\mu (\tilde H(\EuScript{C}_{\tilde H}))=0$ by the Morse-Sard property. Hence $\mu(R_{\tilde H})=1$.
  Then for $R_{\tilde H}'= \tilde H (R_{\tilde H} )$ we have $\tilde \mu(R_{\tilde H}')=1$, where
 as before $\tilde \mu=\tilde H_*(\mu)$ is the measure of maximal entropy for $f$.
Now the  Lusin's N-property of $\tilde H$ yields that $\tilde \mu$ is absolutely continuous and then
equivalent to $\mu$. Hence  we also have $ \mu(R_{\tilde H}')=1$. Since $H= \tilde H^{-1}$ is a
homeomorphism, and $D_x \tilde H$ is invertible for $x\in R_{\tilde H}$, it follows that
$D_y H=(D_x \tilde H)^{-1}$ is the differential of $H$ for each $y=\tilde H (x)$ in  $R_{\tilde H}'$.

Therefore, we can again differentiate $H\circ f=A\circ H$ to obtain \eqref{ABconj2} and then
the conjugacy equation \eqref{ABconj2} with $\c(x)= D_x H$ on the set  $R_{\tilde H}'$ of full
measure for both $\mu$ and $\tilde \mu$.
Then by Theorem~\ref{constant cocycle} applied with $f$ and $\tilde \mu$  we obtain that
$\c(x)= D_x  H $ is H\"older on $\T^N$ and hence so is $\c(y)^{-1}=D_x \tilde H$.
Since  $\tilde H = H^{-1}$ is in $W^{1,q}$ we conclude that $H^{-1}$ is $C^{1+\text{H\"older}}$
diffeomorphism. In this case we also get \eqref{hkDH}.

%%%%%%%%%%%%%%%%%%%%%%%%%%%%%%%

\subsection{Estimating $\|{H-I}\|_{C^{1+\beta}}$} \label{PGE}
We showed that any conjugacy $H$ is a $C^{1+\text{H\"older}}$ diffeomorphism satisfying \eqref{hkDH}.
Now we prove estimate \eqref{C1H est} for the conjugacy $H$ that is $C^0$ close to the identity.

Any two conjugacies  in the homotopy class of the identity differ by a composition with an affine automorphism commuting with $A$, which is translation $T_v(x) = x+v$, where $v\in \T^N $ is a
fixed point of $A$. It is well known that if $f$ is $C^1$-close to $A$, then it has a unique fixed
point $p$ which is the perturbation of $0$. More precisely, there are $0<\delta(A),r(A)<1/5$ and
$k(A)$ so that for each $f$ satisfying $\|{f-A}\|_{C^{1}}<\delta(A)$ there is a unique fixed point
$p=f(p)$ with $d(p,0)<r(A)$ and it satisfies
$$d(p,0)\le k(A)  \|{f-A}\|_{C^{0}}.$$
Since $H$ maps fixed points of $f$ to those of $A$ we see that if $ \|{H-I}\|_{C^{0}}<r(A)$ then
it is in the homotopy class of the identity and satisfies $H(p)=0$.

Replacing $f$ by $\tilde f= T_{-p} \circ f \circ T_p$ we can change $p$ to $0$.
Since for $\tilde f(x)=f(x+p)-p$ we have that
$$\|D \tilde f -A\|_{C^k}=\|D f -A\|_{C^k}\quad\text{for any $k\ge0$,}
$$
and so only $\|{f-A}\|_{C^{0}}$ is affected by this change. Moreover, if we write $f=A+R$,
%where $R:\T^N\to \R^N$,
then
$$
\tilde f(x)-A(x)=A(x+p)+R(x+p)-p-A(x)=R(x+p)+A(p)-p
$$
and hence
$$
\| \tilde f -A\|_{C^0}\le \|  R\|_{C^0}+ \|A(p)-p\|=\| f -A\|_{C^0}+ \|A(p)-f(p)\| \le 2\| f -A\|_{C^0}.
$$
Thus $\|{\tilde f-A}\|_{C^{1+\beta}}\le 2 \|{f-A}\|_{C^{1+\beta}}$. Also, if $\tilde H $ is the corresponding conjugacy between $\tilde f$ and $A$ then $H(x)=\tilde H (x-p)$ and hence
$$\|{H-\Id}\|_{C^{1+\beta}} \le \|{\tilde H-\Id}\|_{C^{1+\beta}} + d(p,0) \le \|{\tilde H-\Id}\|_{C^{1+\beta}} + k(A)  \|{f-A}\|_{C^{0}}
$$
Thus the estimate \eqref{hkDH} for $\tilde H$ via $\tilde f$ would yield the corresponding estimate
for $H$ via $f$. So without loss of generality we will assume that
$$f(0)=0 \;\text{  and }\;H(0)=0.
$$

 Now we recall how the conjugacy equation  $H\circ f=A\circ h$ can be rewritten using lifts.
We denote by $\bar f$ and $\bar H$ the lifts of $f$ and $H$ to $\R^N$ satisfying $\bar f(0)=0 \;\text{  and }\; \bar H(0)=0$ so that we have $\bar H\circ \bar f=A\circ \bar H$ where all maps are $\R^N \to \R^N$.
Since $H$ is homotopic to the identity and $f$ is homotopic to $A$ we can write
$$
\bar H=\Id +  h \quad\text{and}\quad  \bar f=A+ R,
$$ Then the commutation relation on $\R^N$
$$
(\Id +  h)\circ(A + R)=A\circ (\Id + h)\quad\text{yields}\quad
  h= A^{-1} (  h\circ \bar f)+A^{-1}   R.
$$
Since $ h,   R:\R^N \to \R^N$ are $\Z^N$-periodic we can view them as
$$
h=H-\Id :\; \T^N \to \R^N \quad\text{and}\quad  R=f-A  \; :\; \T^N \to \R^N
$$
and rewrite the conjugacy equation as one for $\R^N$-valued functions on $\T^N$
\begin{equation} \label{conj h}
h= A^{-1} (h\circ f)+A^{-1}  R.
\end{equation}

Using the $A$-invariant splitting $\R^N= E^u\oplus E^s$ we define the projections
$h_*$ and $R_*$ of $h$ and $R$ to $E^*$, where $*=s,u$, and obtain
\begin{equation} \label{H_*}
h_*= A_*^{-1} (h_*\circ f)+A^{-1}_*  R_*, \quad \text{where } A_*=A|_{E^*}.
\end{equation}
Thus $h_*$ is a fixed point of the affine operator
\begin{equation} \label{T*}
T_* (\psi )=A_*^{-1} (\psi \circ f)+A^{-1}_*  R_*
\end{equation}
%with the inverse $T_*^{-1} (\phi )=L_*\, (\phi \circ f^{-1})- L_* \, ( G_* \circ f^{-1})$.

Since $\|A_u^{-1}\|<1$, the operator $T_u$ is a contraction on the space $C^0(\T^d,E^u)$,
and thus $h_u$ is its unique fixed point
\begin{equation} \label{Hu}
h_u = \lim_{m\to \infty} T_u^m (0) \,= \sum_{m=0}^\infty A_u^{-m} (A^{-1}_u  R_u \circ f^m).
\end{equation}
Hence
\begin{equation} \label{hu norm}
\|h_u\|_{C^0} \,\le\, \sum_{m=0}^\infty \|A_u^{-1}\|^{m+1}  \|R_u\|_{C^0}
\,\le\, k\,   \|R_u\|_{C^0} \,\le\, k \,  \|A-f\|_{C^0} .
\end{equation}
%Similarly, $T_s^{-1}$ is a contraction on $C^0(\T^d,E^s)$ and $H_s$ is its  unique fixed point
%\begin{equation} \label{Hs} H_s = \lim_{k\to \infty} T_s^{-k} (0) = -\sum_{k=1}^\infty L_s^{k} (G_s \circ f^{-k}). \end{equation}
Similarly, $h_s$ is the unique fixed point of contraction $T_s^{-1}$ and hence satisfies a similar estimate.
Combining them we conclude that
\begin{equation} \label{h norm}
\|H-\Id \|_{C^0}=\|h\|_{C^0} \le  k _0 \|R\|_{C^0} = k_0  \|A-f\|_{C^0} .
\end{equation}

%%%%%%%%%%%%%%%

\vskip .2cm
%We denote by $\mathcal{W}^s$ and $\mathcal{W}^u$ the stable and unstable foliations of $f$. The main part of the proof of Theorem \ref{Holder est} is to establish the following proposition.

Now we estimate  $\|{H-\Id}\|_{C^{1+\beta}}$ using  \eqref{h norm},  \eqref{hkDH}, and  the following elementary interpolation lemma.
We note that $DH=\Id+Dh$, so that $\hk_{\beta}({Dh}) =\hk_{\beta}({DH})$.

\begin{lemma}\label{Int}
If $\,h:\T^N \to \R^N$ satisfies $\hk_{\beta}({Dh})\le K$ then
$$
\| Dh\|_{C^0} \le 8 \,\| h\|_{C^0}^{\beta/(1+\beta)}K^{1/(1+\beta)}.
$$
\end{lemma}
\begin{proof}
Denote $b=\| Dh\|_{C^0}$ and choose $x\in \T^N$ such that $\|D_x h\|=b$. Then for some unit vectors
$u,v\in \R^N$ we have $(D_x h)u=bv.$ For $y\in \T^N$ let $b_y=\langle  (D_y h) u, v \rangle$, so $b_x=b$. Then
$$
|b-b_y| \le \|(D_x h) u -(D_y h) u\| \le K \, d(x,y)^\beta \le b/2 \;\; \text{ if } \;\;  d(x,y) \le (b/2K)^{1/\beta}
$$
and hence $b_y\ge b/2$ for such $y$. Consider $y(t)=x+tu$, with $0\le t \le t_0=(b/2K)^{1/\beta}$,
and $g(t)=\langle h(y(t)), v \rangle$.
Then
$$g'(t)=\langle  (D_y h) u, v \rangle=b_{y(t)} \ge b/2,
$$
and hence by integrating we get $  bt_0/2 \le g(t_0)-g(0)$. Since $|g(t_0)-g(0)| \le 2\| h\|_{C^0}$
we obtain $bt_0\le 4\| h\|_{C^0}$. Substituting $t_0=(b/2K)^{1/\beta}$ we obtain
$$
b(b/2K)^{1/\beta}\le 4\,\| h\|_{C^0} \;\Rightarrow\;
b ^{(1+\beta)/\beta} \le 4 \,\| h\|_{C^0}  (2K)^{1/\beta}
 \;\Rightarrow\; b\le 8 \,\| h\|_{C^0}^{\beta/(1+\beta)}K^{1/(1+\beta)}
$$
as $\,4^{\beta/(1+\beta)}2^{1/(1+\beta)}<8$.
\end{proof}
\vskip.3cm
We denote $a=\|h\|_{C^0}$, $b=\| Dh\|_{C^0}$, and $d=\|{f-A}\|_{C^{1+\beta}}$.
Then
$$\| DH\|_{C^0}=\| \Id +Dh\|_{C^0}\le 1+b,
$$
and hence \eqref{hkDH} implies that
\begin{equation}\label{1+b}
K=\hk_\beta(Dh)=\hk_\beta(DH)\le  k (1+b) d.
\end{equation}
Also, by \eqref{h norm} we have $a=\|h\|_{C^0}\le k_0d$.
Then Lemma \ref{Int}   gives
$$b\le 8 (kd)^{\beta/(1+\beta)}(k (1+b) d)^{1/(1+\beta)}< k_1  \,d  (1+b)^{1/(1+\beta)}.
$$
It follows that $b$ is bounded by some $k_2$ if $d\le1$.
Then \eqref{1+b} implies that
$$K=\hk_\beta(Dh)\le  k_3 d.
$$
 With this $K$ Lemma \ref{Int}  gives
$$b\le 8 (kd)^{\beta/(1+\beta)}(k_3 d)^{1/(1+\beta)} \le k_4 d.
$$
We conclude that
$$b=\| Dh\|_{C^0}< k_4 d, \quad a=\|h\|_{C^0}\le k_0\,d, \;\text{ and }\;\hk_\beta(Dh)\le  k_3  d,
$$
so that
$$\|{H-\Id}\|_{C^{1+\beta}}=\|h\|_{C^{1+\beta}} \le k_5 d= k_5 \,\|{f-A}\|_{C^{1+\beta}}.
$$
This completes the proof of Theorem \ref{HolderConjugacy}.

%%%%%%%%%%%%%%%%%%%%%%%
%%%%%%%%% END proof Technical Theorem

%%%%%%%%%%. Splitting %%%%%%%%%%%%%
%%%%%%%%%%%%%%%%%%%%%%%%%%%%%%%%%%

\section{Linearized conjugacy equation} \label{linearized}

In this section we begin the proof of Theorem \ref{th:4}, and in the next one we will complete it
using an iterative process.
In these sections we fix a hyperbolic matrix $A\in SL(N,\Z)$. We will use $K$ to denote any
constant that depends only on $A$, and $K_{x}$ to denote a constant that also depends on
a parameter $x$.

\subsection{Preliminaries}
 Set $\tilde{A}=(A^\tau)^{-1}$ where $A^\tau$ denotes transpose matrix.
We call $\tilde{A}$ the dual map on $\ZZ^N$.  Since $A$ is hyperbolic so is  $\tilde{A}$,
and we denote its stable and unstable subspaces by $\tilde E^s$ and $\tilde E^u$.
Thus there is $\rho>1$ ( $\rho < \min \{ \rho_{i_0+1}, \rho_{i_0}^{-1}\}$) such that
\begin{align}\label{for:9}
  \norm{\tilde{A}^kv}&\geq K\rho^k\norm{v}, \quad k\geq 0,\;v\in \tilde E^u, \\
  \norm{\tilde{A}^{-k}v}&\geq K\rho^{k}\norm{v}, \quad k\geq 0,\;v\in \tilde  E^s \notag.
\end{align}
%$$1<\rho<\min\, \{\min_{|\lambda_i|>1} |\lambda_i|,\,\min_{|\lambda_i|<1} |\lambda_i|^{-1}\} .$$

For a subspace $V$ of $\RR^N$, we use $\pi_{V}$ to denote the (orthogonal) projection to $V$.
For any integer vector $n\in \Z^N$ we write $n_s= \pi_{\tilde E^s}n$ and $n_u= \pi_{\tilde E^u}n$.
Since $\tilde{A}\in SL(N,\Z)$ is hyperbolic, for any $0\ne n\in \Z^N$ both $n_s$ and  $n_u$ are nonzero
and there is a unique $k_0=k_0(n)\in \Z$ such that
$$
\begin{aligned}
&\| \tilde{A}^kn_s \| \ge \| \tilde{A}^kn_u\|
\quad \text{for all } k\le k_0\quad\text{and} \\
&\| \tilde{A}^kn_s \| < \| \tilde{A}^kn_u\|
\quad \text{for all } k> k_0.
\end{aligned}
$$
The corresponding element $\tilde{A}^{k_0(n)}  n$ on the orbit of $n$ will be called {\em minimal} and
 \begin{equation} \label{Min}
M=\{\tilde{A}^{k_0(n)}  n:\;0 \ne n\in\ZZ^N\} \subset \ZZ^N\backslash 0.
 \end{equation}
For any $n\in M$ we have $\|n_s\| \geq\frac 12 \|n\|$ and $\|\tilde{A}n_u\| >  \frac 12 \|\tilde{A}n\|$.

  \smallskip
For a function $\theta \in L^2(\T^N,\C)$ we denote its Fourier coefficients by $\hat \theta_n$, $n\in\ZZ^N$,
so that
$$\theta(x)=\sum_{n\in\ZZ^N}\widehat{\theta}_ne^{2\pi \mathbf{i} \, n\cdot x} \quad \text{in } L^2(\TT^N).$$
 We say that $\theta$ is \emph{excellent} (for $A$) if $\widehat{\theta}_n=0$ for all $n\notin M$.
 \vskip.2cm

To simplify our estimates, instead of the standard Sobolev spaces we will work the spaces
$\HH^s(\TT^N)$, $s>0$, defined as follows.
A function  $\theta\in L^2(\TT^N)$ belongs to $\HH^s(\TT^N)$ if
      \begin{align*}
       \norm{\theta}_s\overset{\text{def}}{=}\sup_n\,\abs{\widehat{\theta}_n}\norm{n}^s+\abs{\widehat{\theta}_0}<\infty.
      \end{align*}
The following relations hold (see, for example, Section 3.1 of \cite{LLAVE}). If $\sigma>N+1$ and $r\in\NN$, then for any $\theta\in C^r(\TT^N)$ and $\omega\in \HH^{r+\sigma}$
  we have $\theta\in \HH^{r}$ and $\omega\in C^r(\TT^N)$ with estimates
  \begin{align}\label{for:19}
   \norm{\theta}_r\leq K\norm{\theta}_{C^r}\quad \text{and}\quad\norm{\omega}_{C^r}\leq K\norm{\omega}_{r+\sigma}.
  \end{align}

For a vector-valued function $\theta :\T^N \to \C^m$ we denote its coordinate functions by $\theta _j$, $j=1,\dots,m$. We say that $\theta$ is in $\HH^s(\TT^N)$ if each $\theta_j$ is in $\HH^s(\TT^N)$  and set
  \begin{align*}
   \norm{\theta }_s\overset{\text{def}}{=}\max_{1\leq j\leq m}\,\norm{\theta _j}_s,\quad \widehat{\theta }_n\overset{\text{def}}{=}
((\widehat{\theta _1})_n,\dots,(\widehat{\theta _m})_n) \quad \text{  for any $n\in \ZZ^N$}
  \end{align*}
We say that $\theta $ is excellent if $\theta _j$ is excellent for each $j$.

%%%%%%%%%%%%%%%%%%%%%%%%%%%%

\subsection{Twisted cohomological equation over $A$ in high regularity}$\;$

A crucial step in the iterative process is solving the twisted cohomological equation
\begin{equation}
  A\omega-\omega\circ A=\theta%\qquad A_i\omega-\omega\circ A=\theta.
\end{equation}
 over $A$, which can be viewed as the linearized conjugacy equation.
 In this section we give preliminary results on solving this equation in high regularity.
 We start with a scalar cohomological equation over $A$ twisted by  $\lambda\in\CC\backslash \{0,1\}$,
\begin{align} \label{lambda twist}
  \lambda\omega-\omega\circ A=\theta.
\end{align}
The next lemma  shows that the obstructions to solving it in $C^\infty$ category
 are sums of Fourier coefficients of $\theta$ along the orbits of $\tilde{A}$. Moreover, for any $C^\infty$
 function $\theta$ there is a well behaved splitting $\theta=\theta^\iota+\theta^*$, where $\theta^\iota$ can be view as a projection to the space of twisted coboundaries and $\theta^*$ as the error.
%We note that both $\theta^\iota$ and $\theta^*$ are of the same order as $\theta$.
A similar result was proved for ergodic toral automorphisms in \cite{Damjanovic4} and used for establishing $C^\infty$ local rigidity of some partially hyperbolic $\ZZ^k$ actions.
We prove the result for hyperbolic case to keep our exposition self-contained and  get a better constant $\sigma (\lambda)$.

\begin{lemma}\label{le:3} For a function $\theta:\TT^N\to \CC$  in $\HH^{a}(\T^N)$ and $\lambda\in\CC\backslash  \{0,1\}$ we define
   \begin{align*}
  D_\theta(n)=\sum_{i=-\infty}^\infty\lambda^{-(i+1)}\widehat{\theta }_{\tilde{A}^{i}n}.
\end{align*}
Suppose $a\geq\sigma(\lambda)=\frac{|\log |\lambda| |}{\log\rho}+1$,  where $\rho>1$ is the expansion rate of $\tilde{A}$ from \eqref{for:9}. Then
\begin{itemize}
  \item[{\bf (i)}] The sum $D_\theta(n)$
converges absolutely for any $n\neq0$; moreover the function
\begin{align*}
 \theta^*\overset{\text{def}}{=}\sum_{n\in M}D_\theta(n)e^{2\pi \mathbf{i} \, n\cdot x},
\end{align*}
 where $M$ is from \eqref{Min}, is in $\HH^{a}(\T^N)$ with the estimate $\norm{\theta^*}_{a}\leq K_{a,\lambda}\norm{\theta}_{a}.$

\smallskip

\item[{\bf (ii)}] If $D_\theta(n)=0$ for any $n\neq0$, then the equation \eqref{lambda twist}
 has a solution $\omega\in\HH^{a}(\T^N)$ with the estimate
\begin{align*}
       \norm{\omega}_{a}&\leq K_{r,\lambda}\norm{\theta}_{a}.
      \end{align*}
%\begin{align*} \abs{S(n)}\leq C_{a,\delta}\mathfrak{b}\norm{n}^{-a} \end{align*}

  \smallskip

  \item[{\bf (iii)}] If  the equation \eqref{lambda twist}
has a solution $\omega\in\HH^{\sigma(\lambda)}(\T^N)$, then $D_\theta(n)=0$ for any $n\neq0$.

%\begin{align*} \abs{S(n)}\leq C_{a,\delta}\mathfrak{b}\norm{n}^{-a} \end{align*}

  \smallskip

  \item[{\bf (iv)}]
For $\theta^\iota\overset{\text{def}}{=}\theta-\theta^*$ the equation:
      \begin{align*}
 \lambda\omega-\omega\circ A=\theta^\iota
\end{align*}
has a solution $\omega\in\HH^{a}(\T^N)$ with the estimate $\; \norm{\omega}_{a}\leq K_{r,\lambda}\norm{\theta}_{a}. $
\end{itemize}

\end{lemma}

\begin{remark}\label{remark:1} We emphasize that the existence of $\theta^*$ requires a high regularity of $\theta$. In fact, for any $b\leq \sigma (\lambda)$, we have to estimate it as
$\norm{\theta^*}_{b}\leq K_{\lambda}\norm{\theta}_{\sigma(\lambda)}.$
\end{remark}

\begin{proof} We define
$$
D_\theta(n)_+=\sum_{i\geq 1}\lambda^{-(i+1)}\widehat{\theta }_{\tilde{A}^{i}n}
 \quad\text{and}\quad
 D_\theta(n)_-=-\sum_{i\leq 0}\lambda^{-(i+1)}\widehat{\theta }_{\tilde{A}^{i}n}.
$$
{\bf (i)}.
Let  $n\in M$. The inequality $\|\pi_{\tilde E^s}(n)\|\ge \frac 12 \|n\|$ we obtain
\begin{align}\label{for:26}
 |D_\theta(n)_-|&\,\leq\, \norm{\theta}_a\sum_{i\leq 0}\,|\lambda|^{-(i+1)}\,\norm{\tilde{A}^in}^{-a}\leq\, \norm{\theta}_a\sum_{i\leq 0}\,|\lambda|^{-(i+1)} \,\norm{\pi_{\tilde E^s}(\tilde{A}^in)}^{-a}\notag\\
 &\leq\,
  \norm{\theta}_aC^{-a}\sum_{i\leq 0}\,|\lambda|^{-(i+1)}\rho^{ia}\,\norm{\pi_{\tilde E^s}(n)}^{-a}\overset{\text{(1)}}{\leq} \,K_{a,\lambda}\norm{\theta}_a\,\norm{n}^{-a}.
\end{align}
Here in $(1)$ convergence is guaranteed by $a>{\text {\small $\frac{|\log |\lambda| |}{\log\rho}$}}$.
The sum $D_\theta(n)_+$ can be estimated similarly using the inequality
$\|\pi_{\tilde E^u}(\tilde{A}n)\|\ge \frac 12 \|\tilde{A}n\|$. Hence we get
\begin{align*}
       \norm{\theta^*}_{a}&\leq K_{a,\lambda}\norm{\theta}_{a}.
      \end{align*}
For any $z\in \ZZ^N$ and $k\in \ZZ$, we see that
\begin{align}\label{for:29}
 D_\theta(\tilde{A}^kz)=\lambda^{k}D_\theta(z).
\end{align}
This shows that $D_\theta(n)$
converges absolutely for any $n\neq0$. %Hence we finish the proof.

\vskip.2cm

{\bf (ii)} In the dual space the equation $\lambda\omega-\omega\circ A=\theta$ has he form
\begin{align*}
 \lambda\widehat{\omega}_n-\widehat{\omega}_{\tilde{A}n}=\widehat{\theta}_n,\qquad \forall\,n\in\ZZ^N.
\end{align*}
For $n=0$,  we let $\widehat{\omega }_0=\frac{\widehat{\theta }_0}{\lambda-1}$. For any $n\neq0$, let $\widehat{\omega }_n=D_\theta(n)_-$. Then
$\omega=\sum_{n\in \ZZ^N}\widehat{\omega }_ne^{2\pi \mathbf{i}}$
 is a formal solution. Next, we obtain its Sobolev estimates.
 If $\|\pi_{\tilde E^s}(n)\|\ge \frac 12 \|n\|$, then
from \eqref{for:26} we have
\begin{align}\label{for:27}
 |\widehat{\omega }_n|\cdot \norm{n}^{a}\leq K_{a,\lambda}.
\end{align}
If $\|\pi_{\tilde E^u}(\tilde{A}n)\|\ge \frac 12 \|\tilde{A}n\|$, then the assumption $D_\theta(n)=0$ implies that $\widehat{\omega }_n=D_\theta(n)_+$. The arguments in {\bf (i)}
show that \eqref{for:27} still holds. %Hence we get the result.

\vskip.2cm
{\bf (iii)} By {\bf (i)} and \eqref{for:29} we have: for any $n\neq0$
\begin{align*}
 D_\theta(n)=D_{\lambda\omega-\omega\circ A}(n)=\lambda D_{\omega}(n)-D_{\omega}(\tilde{A}n)=\lambda D_{\omega}(n)-\lambda D_{\omega}(n)=0.
\end{align*}

\vskip.2cm
{\bf (iv)} It is clear that $D_{\theta^\iota}(n)=D_{\theta-\theta^*}(n)=D_{\theta}(n)-D_{\theta^*}(n)=0$ for any $n\neq0$. Then the result follows from {\bf (ii)}.
\end{proof}

Now we extend Lemma \ref{le:3} to the vector valued case. We consider  the equation
$$A_i\omega-\omega\circ A=\theta $$
with the twist given by the restriction $A_i=A|E^i$,
where  $E^i$, $i=1,\dots ,L$,\, is a subspace of the splitting \eqref{splitL}.
 We note that any eigenvalue $\lambda$ of $A_i$ satisfies $|\lambda|=\rho_i$.

\begin{lemma}\label{le:2} Let  $\rho>1$ be  the expansion rate for $\tilde{A}$ from \eqref{for:9} and let
\begin{equation}\label{sigma}
\sigma=\max_{i=1,\dots ,L} \left( {\text {\small $\frac{|\log \rho_i |}{\log\rho}$}}+1\right)N+N+2.
\end{equation}
Then for any $i=1,\dots ,L$ and any $C^\infty$ map $\theta:\TT^N\to \CC^{N_i}$, there is a splitting of $\theta$
  \begin{align*}
   \theta=\theta^\iota+\theta^*
  \end{align*}
  such that the equation:
      \begin{align}\label{for:1}
 A_i\omega-\omega\circ A=\theta^\iota
\end{align}
has a $C^\infty$ solution $\omega$ with estimates
\begin{align*}
       \norm{\omega}_{C^r}\leq K_r\norm{\theta}_{C^{r+\sigma }},\qquad \forall\,r\geq0;
      \end{align*}
  and $\theta^*:\TT^N\to \CC^{N_{i}}$ is an excellent $C^\infty$ map so that  for any $r\geq0$
      \begin{align*}
       \norm{\theta^*}_{C^r}\leq K_{r}\norm{\theta}_{C^{r+\sigma}}\qquad \text{and} \qquad
       \norm{\theta^*}_{r}\leq K_{r}\norm{\theta}_{r+\sigma-2-N}.
      \end{align*}
\end{lemma}

\begin{proof} If $A_i$ is semisimple, then the conclusion follows directly from
Lemma \ref{le:3} as the equation \eqref{for:1}  splits into finitely many equations of
the type
     \begin{align*}
 \lambda_j\omega_j-\omega_j\circ A=(\theta_j)^\iota
\end{align*}
where $\theta_j$ is a coordinate function of $\theta$ and $\lambda_j$ is the corresponding eigenvalue of $A_i$.

If $A_i$ is not semisimple, we choose a basis in which $A_i$ is in its Jordan normal form
with some nontrivial Jordan blocks. We note that the excellency of maps is preserved under
the change of basis. Let $J=(J_{l,j})$ to be an $m\times m$  Jordan block of $A_i$ corresponding
to an eigenvalue $\lambda$ with $|\lambda|=\rho_i$, that is, $J_{l,l}=\lambda$ for all $1\leq l\leq m$ and $\lambda_{l,l+1}=1$ for all $1\leq l\leq m-1$. Then equation \eqref{for:1} splits into equations
of the form
\begin{align}\label{for:4}
  J\Omega-\Omega\circ A=\Theta^\iota,
\end{align}
corresponding to the Jordan blocks $J$.  Each equation \eqref{for:4} further splits into the following $m$ equations:
    \begin{align*}
 \lambda\Omega_j-\Omega_j\circ A+\Omega_{j+1}&=(\Theta^\iota)_j, \qquad\text{and} \\
 \lambda\Omega_m-\Omega_m\circ A&=(\Theta^\iota)_m=(\Theta_m)^\iota,
\end{align*}
$1\leq j\leq m-1$. For the $m$-th equation, Lemma \ref{le:3} gives the splitting
\begin{align*}
   \Theta_m=\lambda\Omega_m-\Omega_m\circ A+(\Theta^*)_m
  \end{align*}
  where $\Omega_m$, $(\Theta^*)_m=(\Theta_m)^*$,  and $(\Theta^\iota)_m=\lambda\Omega_m-\Omega_m\circ A$ are $C^\infty$ functions
satisfying the estimates:
 \begin{align*}
    \max\{\norm{(\Theta^*)_m}_{r},\,\norm{\Omega_m}_{r}\}&\leq K_{r,m}\norm{\Theta}_{r+\sigma(\rho_i)},\qquad \forall\,r\geq0
      \end{align*}
      and $\Theta_m^*$ is excellent.

      Now we proceed by induction. Fix $1\leq k\leq m-1$ and assume that
for all $k+1\leq j\leq m$ we already have the splitting
\begin{align*}
   \Theta_j=\lambda\Omega_j-\Omega_j\circ A+\Omega_{j+1}+(\Theta^*)_j
  \end{align*}
   where $\Omega_j$, $\Theta_j^*$,  and $(\Theta^\iota)_j=\lambda\Omega_j-\Omega_j\circ A+\Omega_{j+1}$ are $C^\infty$ functions satisfying the estimates:
 \begin{align}\label{for:2}
       \max\{\norm{\Omega_j}_{r},\,\norm{(\Theta^*)_j}_{r}\}&\leq K_{r,j}\norm{\Theta}_{r+(m-j+1)\sigma(\rho_i)},\qquad \forall\,r\geq0
      \end{align}
      and $(\Theta^*)_j$ is excellent. By Lemma \ref{le:3} we obtain the splitting
      \begin{align*}
   \Theta_k-\Omega_{k+1}=\lambda\Omega_k-\Omega_k\circ A+(\Theta_k-\Omega_{k+1})^*
  \end{align*}
   where $\Omega_k$, $(\Theta^*)_k=(\Theta_k-\Omega_{k+1})^*$,  and $(\Theta^\iota)_k=\lambda\Omega_k-\Omega_k\circ A+\Omega_{k+1}$ are $C^\infty$ functions satisfying the estimates following from \eqref{for:2}:
 \begin{align*}
       \max\{\norm{\Omega_k}_{r},\,\norm{(\Theta^*)_k}_{r}\}&\leq K_{r}\norm{\Theta_k-\Omega_{k+1}}_{r+\sigma(\rho_i)}\leq K_{r,k}\norm{\Theta}_{r+(m-k+1)\sigma(\rho_i)}, \quad \forall\,r\geq0
      \end{align*}
      and $(\Theta^*)_k$ is excellent. Let $\Omega$, $\Theta^\iota$ and $\Theta^*$ be maps with coordinate functions $\Omega_j$, $(\Theta^\iota)_j$ and $(\Theta^*)_j$, $1\leq j\leq m$ respectively.  Hence we show that
      there is a splitting of $\Theta$
  \begin{align*}
   \Theta=\Theta^\iota+\Theta^*
  \end{align*}
  such that the equation \eqref{for:4} has a $C^\infty$ solution $\Omega$ with estimates.
\begin{align*}
       \max\{\norm{\Theta^*}_{r},\,\norm{\Omega}_{r}\}&\leq K_{r}\norm{\Theta}_{r+m\sigma(\rho_i)}, \quad \forall\,r\geq0
      \end{align*}
      This can be
repeated for all corresponding blocks of $A$. Since the maximal size of a
Jordan block is bounded by $N$, we obtain estimates for the $\norm{\cdot}_r$ norms of $\omega$ and $\theta^*$. This implies estimates for the $\norm{\cdot}_{C^r}$ norms
as well by \eqref{for:19}.
\end{proof}

%%%%%%%%%%%%%%%%%%%%%%%%%%%
%%%%%%%%%%%%%%%%%%%%%%%%%%%
\subsection{Main result on the linearized equation.}
%%%%%%%%%%%%%%%%%%%%%%%%%%%
 The next theorem is our main result on solving the linearized equation.
It  plays the crucial role in the inductive step of the iterative process, Proposition \ref{po:1}.
The goal of the inductive step is, given a $C^{1}$ conjugacy $H$ between $A$ and its perturbation $f$, to construct a smaller perturbation $\tilde f$ which is smoothly conjugate to $f$ by $\tilde H$. The conjugacy $\tilde H$ is constructed in the form $\tilde H = I - \omega$, where $\omega$ is a $C^\infty$ approximate solution of the linearized equation given by Theorem \ref{th:3}.
The $C^{1}$ conjugacy $H$ is upgraded to $C^{1+a}$ by Theorem \ref{HolderConjugacy}. It yields
an {\it approximate}  $C^{1+a}$ solution $\mathfrak{h}=H-I$ of the linearized equation \eqref{for:11}. This necessitates the introduction of the error term $\Psi$ in the assumption of the theorem.

\begin{theorem}\label{th:3} Let $A$ be weakly irreducible hyperbolic automorphism of $\T^N$. Suppose that
\begin{align}\label{for:11}
A \mathfrak{h}-\mathfrak{h}\circ A=\mathcal{R}+\Psi,
\end{align}
where maps $\mathfrak{h},\Psi :\T^N\to \R^N$ are  $C^{1+a}$  and $\mathcal{R}:\T^N\to \R^N$ is $C^{\infty}$.
\vskip.1cm

Then there exist  $C^\infty$ maps $\omega, \Phi :\T^N\to \R^N$
satisfying the equation
\begin{align}\label{for1}
 \mathcal{R}=A \omega-\omega\circ A+\Phi
\end{align}
and the estimates
\begin{gather*}
       \norm{\omega}_{C^r}\leq K_r\norm{\mathcal{R}}_{C^{r+\sigma}}\\
       \norm{\Phi}_{C^0}\leq K_{l,a}(\norm{\Psi}_{C^{1+a}})^{\frac{l-2-N}{l+N}}(\norm{\mathcal{R}}_{C^{l+\sigma}})^{\frac{2N+2}{l+N}}
  %     \norm{\Phi}_{C^0,\Delta_2}\leq C_{r,a}(\norm{\Psi}_{C^{a},\Delta_2})^{\frac{r-2-N}{r}}(\norm{\mathcal{R}}_{C^{r+\sigma}})^{\frac{N+2}{r}}
      \end{gather*}
for any $r\geq0$ and $l>N+2$, where $\sigma$ is given by \eqref{sigma}.

\end{theorem}
For traditional KAM iteration scheme, the convergence requires the error $\Phi$ in solving the twisted coboundary \eqref{for1} to be small compared with $\mathcal{R}$.
This is established by showing that $\Phi$ is tame with respect to $\Psi$, which is almost quadratically small with respect to $\mathcal{R}$. Tameness means that the $C^r$ norm
of $\Phi$ can be bounded by the $C^{r+p}$ norm of $\Psi$, where $r$ is arbitrarily large while $p$ is a constant.

One difficulty  in our setting is that the estimate of $\Phi$ depends on $\Psi$ and $\mathcal{R}$ rather than on $\Psi$ only.  This results in technical issues in proving convergence of the iterative
procedure, and so the traditional KAM scheme fails to work. We resolve this issue by introducing a parameter $l$ when estimating $\norm{\Phi}_{C^0}$. If the parameters are well chosen, the
constructed approximation behaves as if it were tame.

The main difficulty in estimating $\Phi$ in our setting is  that low regularity of $\mathfrak{h}$ yields smallness of $\Psi$ only in $C^{1+\text{H\"{o}lder}}$ norm, see Lemma \ref{le:7} and equation \eqref{for2}. This does not allow us to directly estimate  orbit sums of Fourier
coefficients   and split $R$ into a smooth coboundary $R^\iota=A \omega-\omega\circ A$
and an error term $R^*=\Phi$, see Remark \ref{remark:1}.
To overcome this problem we use the splitting $\R^N=\oplus E^i$ to decompose the equation \eqref{for:11} and then differentiate $i^{th}$ component along directions in $E^i$.
This allows us to ``balance" the twist (up to a polynomial growth of Jordan blocks) and
 analyze the differentiated equation using H\"older regularity. This is done in the following
 Lemma \ref{cor:2}.
After that, we establish Lemma \ref{le:4} to relate Fourier coefficients of a function and its
directional derivatives. We then complete the proof of Theorem \ref{th:3} in Section \ref{proof th:3}.

\vskip.3cm

Now we begin the analysis of the differentiated equation \eqref{for:11}. For any $1 \le i\le L$ and any unit vector $u_0 \in E^i$, we consider unit vectors $u_k$ and scalars $a_k$, $k\in \Z$,
given by
\begin{equation}\label{u_k}
u_k= \text{\small $\frac{A_i^ku_0}{\norm{A_i^ku_0}}$} \quad \text{and} \quad
a_k=\norm{A_iu_k}= \text{\small $\frac{\norm{A_i^{k+1}u_0}}{\norm{A_i^ku_0}}$}\quad \text{so that} \quad
A_iu_k=a_ku_{k+1}.
\end{equation}
We define a sequence of matrices $P_k\in GL(N_i,\R)$ which commute with $A_i$ and satisfy the recursive equation
\begin{equation}\label{P_k}
P_{k+1}=a_kA_i^{-1}P_k.
\end{equation}
Specifically, we set
\begin{align}\label{for:20}
P_0=\Id \quad\text{and}\quad P_k=&\left\{\begin{aligned} &\,a_0\cdots a_{k-1}A_i^{-k}=
\norm{A_i^{k}u_0}\,A_i^{-k},&\quad &k>0,\\
&\,(a_{-1}\cdots a_{-k})^{-1}A_i^{k}=\norm{A_i^{-k}u_0}\,A_i^{k},&\quad &k<0.
\end{aligned}
 \right.
\end{align}

%%%%

\begin{lemma}\label{cor:2}
Let $\varphi_k: \TT^N\to \R^{N_i}$ be a sequence of maps in $\HH^{a}(\T^N)$, $a>0$, satisfying $\norm{\varphi_k}_a\leq \mathfrak{b}$ for all $k\in\ZZ$, let $P_k\in GL(N_i,\R)$ be as in \eqref{for:20}, and let
   \begin{align*}
  S(n)=\sum_{k\in\Z}P_k\,(\widehat{\varphi_k})_{\tilde{A}^kn}.
\end{align*}
\begin{itemize}
  \item[{\bf (i)}]\label{for:15} For any $n\in M$ the sum $S(n)$
converges absolutely in $\C^{N_i}$ with the estimate \\
$\norm{S(n)}\leq K_{a}\mathfrak{b}\,\norm{n}^{-a}$.%\begin{align*} \abs{S(n)}\leq C_{a,\delta}\mathfrak{b}\norm{n}^{-a} \end{align*}
  \smallskip

  \item[{\bf (ii)}]\label{for:16}
  If $\mathfrak{h}_k: \TT^N\to \R^{N_i}$ is another sequence in $\HH^{a}(\T^N)$
so that  for all $k\in\ZZ$ we have $\norm{\mathfrak{h}_k}_a\leq \mathfrak{c}$ and %$h_k$ and $\varphi_k$ satisfy the twisted equations
\begin{align}\label{for:13}
A_i \mathfrak{h}_k-a_k\mathfrak{h}_{k+1}\circ A=\varphi_k,
\end{align}
then $S(n)=0$ for every $n\in M$.
\end{itemize}
\end{lemma}

\begin{proof}
{\bf (i)}.  Since all eigenvalues of $A_i$ have the same modulus $\rho_i$, we have \eqref{jordan}, and so there exists a constant $C$ such that all $P_k$ satisfy the polynomial estimate
\begin{align}\label{for:3}
\norm{P_k}\le \norm{A_i^{k}}\cdot \norm{A_i^{-k}}\leq C(\abs{k}+1)^{2N}=:p(|k|),\quad \text{for all }k\in \Z.
\end{align}

Let $n\in M.$  We write $\varphi_k=(\varphi_{k,1},\cdots,\varphi_{k,N_i})$ and set
$$
S(n)_+=\sum_{k\geq 1}P_k\,(\widehat{\varphi_k})_{\tilde{A}^kn}
 \quad\text{and}\quad
 S(n)_-=\sum_{k\leq 0}P_k\,(\widehat{\varphi_k})_{\tilde{A}^kn}.
$$
Using the assumption $\norm{\varphi_k}_a\leq \mathfrak{b}$, estimates \eqref{for:3} and \eqref{for:9}, and the inequality $\|\pi_{\tilde E^s}(n)\|\ge \frac 12 \|n\|$ we obtain
\begin{align*}
 \|S(n)_ -\|&\,\leq\,\sum_{k\leq 0}\,\norm{P_k}\max_{1\leq j\leq m}|(\widehat{\varphi_{k,j}})_{\tilde{A}^kn}|
 \,\leq\, \sum_{k\leq 0}\,\norm{\varphi_k}_a\,\norm{P_k}\,\norm{\tilde{A}^kn}^{-a}\\
 &\leq\, \mathfrak{b}\sum_{k\leq 0} \,p(|k|)\,\norm{\pi_{\tilde E^s}(\tilde{A}^kn)}^{-a}\,\leq\,
  \mathfrak{b}C^{-a}\sum_{k\leq 0}\, p(|k|)\,\rho^{ka}\,\norm{\pi_{\tilde E^s}(n)}^{-a} \\
& \leq \,K_{a}\mathfrak{b}\,\norm{n}^{-a}.
\end{align*}
The sum $S(n)_+$ can be estimated similarly using the inequality
$\|\pi_{\tilde E^u}(\tilde{A}n)\|\ge \frac 12 \|\tilde{A}n\|$.

\vskip.2cm

{\bf (ii)} Let  $n\in M$. From the equation \eqref{for:13} we obtain that for any $k\in \Z$
$$
P_k \,  \varphi_k \circ A^{k}=P_k A_i \,  \mathfrak{h}_k \circ A^{k}-a_k P_k \, \mathfrak{h}_{k+1}\circ A^{k+1}
$$
Summing from $-m$ to $j$ and observing that the sum on the right is telescoping as $a_k P_k = A_i P_{k+1}=P_{k+1} A_i$ by the choice of $P_k$ in \eqref{P_k},  we obtain
\begin{align*}
\sum_{k=-m}^j P_k\; \varphi_k\circ A^{k}=  A_iP_{-m}\mathfrak{h}_{-m}\circ A^{-m}-a_jP_j \, \mathfrak{h}_{j+1}\circ A^{j+1}.
\end{align*}
Taking Fourier coefficients and noting that $(\widehat{\theta\circ A^k})_n=\widehat{\theta}_{\tilde{A}^k n}$
we obtain
\begin{align*}
\sum_{k=-m}^j P_k(\widehat{\varphi_k})_{\tilde{A}^{k}n}= A_iP_{-m}(\widehat{\mathfrak{h}_{-m}})_{\tilde{A}^{-m}n}-a_j P_j(\widehat{\mathfrak{h}_{j+1}})_{\tilde{A}^{j+1}n}.
\end{align*}
Since the series
$\sum_{k\in \Z} P_k(\widehat{\mathfrak{h}_k})_{\tilde{A}^{k}n}$ converges by part (i), we have $P_k(\widehat{\mathfrak{h}_k})_{\tilde{A}^{k}n}\to 0$ as $k\to \pm \infty$ and hence,
as $a_k$ are bounded,
 \begin{align*}
   &a_j P_j (\widehat{\mathfrak{h}_{j+1}})_{\tilde{A}^{j+1}n}\to 0,\qquad \text{as }j\to\infty;\quad\text{and}\\
   &A_iP_m(\widehat{\mathfrak{h}_{m}})_{\tilde{A}^{m}n}\to 0,\qquad \text{as }m\to-\infty.
  \end{align*}
 We conclude that $S(n)=0$.
\end{proof}

%%%%%%%%%%%%%%%%%%%%%%%%

\subsection{Directional derivatives}
 In this section we establish some estimates for Fourier coefficients  of a
$C^1$ function  $\theta :\T^N \to \R$ via Fourier coefficients  of its directional derivatives
along a subspace $E^i$ of the splitting \eqref{splitL}. This relies on weak irreducibility of $A$.

For any $v\in\R^N$ with $\norm{v}=1$, we denote the directional derivative of $\theta$ along $v$
by $\theta_v$.

\begin{lemma}\label{le:4} Let $A$ be a weakly irreducible integer matrix and let $v_{i,j}$, $j=1,\dots ,N_i$, be an orthonormal basis of a subspace $E^i$ from  \eqref{splitL}.  Then there exists a constant $K=K(A)$ such that for any $i=1, \dots, L$ and any $C^1$ function $\theta: \TT^N \to \R$,
\begin{align*}
 |\hat{\theta}_n|\,\leq\, K\sum_{j=1}^{N_i}|(\widehat{\theta_{v_{i,j}}})_n|\cdot \norm{n}^N\quad \text{for all }\,n\in\Z^N\backslash 0.
\end{align*}
\end{lemma}

\begin{proof}
 We denote by $\|.\|$ the standard Euclidean norm in $\R^N$. Since $\theta$ is $C^1$, we have
\begin{align*}
 2\pi\textrm{i}(n\cdot v_{i,j})\hat{\theta}_n=(\widehat{\theta_{v_{i,j}}})_n,\qquad 1\leq j\leq N_i.
\end{align*}
Adding over $j$ we obtain that for any  $n\in\ZZ^N\backslash 0$ we have
\begin{align*}
 |\hat{\theta}_n| \,=\,\frac{\sum_{j=1}^{N_i}|(\widehat{\theta_{v_{i,j}}})_n|}{2\pi\sum_{j=1}^{N_i}|n\cdot v_{i,j}|} \,\le\, \frac{\sum_{j=1}^{N_i}|(\widehat{\theta_{v_{i,j}}})_n|}{2\pi \norm{\pi_{{E}^i}n}},
\end{align*}
since for an orthonormal basis $v_{i,j}$ we have $\sum_{j=1}^{N_i}|n\cdot v_{i,j}| \ge  \norm{\pi_{{E}^i}n}$.
 Since $\norm{\pi_{{E}^i}n}=d(n,(\E^i)^\perp)$, to complete the proof it remains to show that
$d(n,(\E^i)^\perp)\geq K' \norm{n}^{-N}$.

Since $A$  is weakly irreducible, so is the transpose $A^\tau$. This follows from
 Lemma \ref{Weak irred} which gives an equivalent condition for weak irreducibility
 in terms of the characteristic polynomial.
We denote the splitting \eqref{splitL} for $A^\tau$ by $\R^N=E^1_\tau \oplus \dots \oplus E^L_\tau$
and similarly let $\hat E^i_\tau = \oplus_{j\ne i} E^i_\tau$. Then we obtain $(\E^i)^\perp=\hat E^i_\tau$.
Indeed, the polynomial
$$p_i(x)=\prod _{|\lambda |=\rho_i}(x-\lambda)^N,$$ where the product is over all eigenvalues of $A$ of modulus $\rho_i$, is real and
$$(\E^i)^\perp=(\ker p_i(A))^\perp = range(p_i(A)^\tau)= range(p_i(A^\tau)) =\hat E^i_\tau,$$
since %$E^i_\tau =\ker p_i(A^\tau)$ and
 $p_i(A^\tau)$ is invertible on $\hat E^i_\tau$. Now the desired inequality
$$d(n,(\E^i)^\perp)=d(n,\hat E^i_\tau) \geq K' \norm{n}^{-N}$$
 follows from Katznelson's Lemma below.
We apply it to $A^\tau$ with the invariant splitting  $\R^N= \hE^i_\tau \oplus E^i_\tau$ and note that
$\hE^i_\tau \cap\Z^N=\{0\}$ by weak irreducibility of $A^\tau$.
\end{proof}

\begin{lemma}[Katznelson's Lemma]\label{le:1} Let $A$ be an $N\times N$ integer matrix.
Assume that $\R^N$ splits as $\R^N=V_1\bigoplus V_2$ with $V_1$ and $V_2$ invariant under $A$
and such that $A|_{V_1}$ and $A|_{V_2}$ have no common eigenvalues. If $V_1\cap\Z^N=\{0\}$, then
there exists a constant $K$ such that
\begin{align*}
d(n,V_1)\geq K\norm{n}^{-N} \quad\text{for all } 0\ne n\in\Z^N,
\end{align*}
where $\norm{v}$ denotes
Euclidean norm and $d$ is Euclidean distance.
\end{lemma}
See e.g. \cite[Lemma 4.1] {Damjanovic4} for a proof.

%%%%%%%%%%%%%%%%%%%%%%%%%

\subsection{Proof of Theorem \ref{th:3} } \label{proof th:3}
Using the splitting $\R^N=\oplus E^i$ we decompose  \eqref{for:11} into equations
\begin{align}\label{for:12}
A_i \mathfrak{h}_i-\mathfrak{h}_i\circ A=\mathcal{R}_i+\Psi_i,\qquad  i=1, \dots, L
\end{align}
where $\mathfrak{h}_i$, $\mathcal{R}_i$ and $\Psi_i$ are coordinate maps in the
of $\mathfrak{h}$, $\mathcal{R}$ and $\Psi$ respectively.

By Lemma \ref{le:2} there is an excellent $C^\infty$ map $\mathcal{R}_i^*$ with estimates
      \begin{align}\label{for:23}
       \norm{\mathcal{R}_i^*}_{C^r}\leq K_r\,\norm{\mathcal{R}_i}_{C^{r+\sigma}},\quad \norm{\mathcal{R}_i^*}_{r}\leq K_{r}\,\norm{\mathcal{R}_i}_{r+\sigma-N-2}
      \end{align}
      for any $r\geq0$, such that the equation:
      \begin{align}\label{for:18}
 A_i\omega_i-\omega_i\circ A=\mathcal{R}_i+\mathcal{R}_i^*
\end{align}
has a $C^\infty$ solution $\omega_i$ with estimates
\begin{align*}
       \norm{\omega_i}_{C^r}\leq K_r\norm{\mathcal{R}_i}_{C^{r+\sigma}},\qquad \forall\,r\geq0.
      \end{align*}
 Let $\omega$ be the map with coordinate maps $\omega_i$.

We obtain from \eqref{for:12} and \eqref{for:18} that $C^{1+a}$ maps $\mathfrak{p}_i=\mathfrak{h}_i-\omega_i$ and $\Lambda_i=-\mathcal{R}_i^*+\Psi_i$ satisfy
\begin{align*}
A_i \mathfrak{p}_i-\mathfrak{p}_i\circ A=\Lambda_i.
\end{align*}

We fix $1 \le i \le L$ and an orthonormal basis $v_{i,j}$ of $E^i$.  We fix $1 \le j \le N_i$
and, as in \eqref{u_k}, consider unit vectors $u_0=v_{i,j}$ and $u_k=\frac{A^ku_0}{\norm{A^ku_0}}$, and let $a_k=\norm{Au_k}$, $k\in\ZZ$. Taking the derivative of the previous equation in the direction of $u_k$
we obtain equations
\begin{align*}
  A_i (\mathfrak{p}_i)_{u_k}-a_{k}(\mathfrak{p}_i)_{u_{k+1}}\circ A=(\Lambda_i)_{u_k},\qquad \forall\,k\in\ZZ.
\end{align*}
We note that for any $k\in\ZZ$ the maps $(\mathfrak{p}_i)_{u_k}$ and $(\Lambda_i)_{u_k}$ are in $C^{a}$ and hence in $\H^a$, as we recall that for any function $g$ by \eqref{for:19} we have
\begin{align}\label{for:21}
 \norm{g_{u_k}}_{a}&\leq K\norm{g_{u_k}}_{C^a}\leq K_1\norm{g}_{C^{1+a}}.
 \end{align}
%where $g$ stands for $\mathfrak{p}_i$, $\Lambda_i$ or $\Psi_i$.

Now we use (ii) of Lemma \ref{cor:2} with $\mathfrak{h}_k=(\mathfrak{p}_i)_{u_k}$,
$\varphi_k=(\Lambda_i)_{u_k}$, and $P_k$ is as defined in \eqref{for:20} to obtain that for any $n\in \mathcal{M}$
\begin{align*}
 \sum_{k\in\ZZ}P_k\widehat{((\Psi_i)_{u_k})}_{\tilde{A}^kn}-\sum_{k\in\ZZ}P_k\widehat{((\mathcal{R}_i^*)_{u_k})}_{\tilde{A}^kn}
 =\sum_{k\in\ZZ}P_k\widehat{((\Lambda_i)_{u_k})}_{\tilde{A}^kn}=0.
\end{align*}
Since $(\mathcal{R}_i^*)_{u_k}$ is excellent, for each $k\in\ZZ$ we have
\begin{align*}
 \sum_{k\in\ZZ}P_k\widehat{((\Psi_i)_{u_k})}_{\tilde{A}^kn}=\sum_{k\in\ZZ}P_k\widehat{((\mathcal{R}_i^*)_{u_k})}_{\tilde{A}^kn}=\widehat{((\mathcal{R}_i^*)_{u_0})}_{n}
\end{align*}
for any $n\in M$, which gives
\begin{align*}
 |\widehat{((\mathcal{R}_i^*)_{u_0})}_{n}|\overset{\text{(1)}}{\leq} K_{a}\max_{k\in\ZZ}\{\norm{(\Psi_i)_{u_k}}_a\}\norm{n}^{-a}\overset{\text{(2)}}{\leq} K_{a,1}\norm{\Psi_i}_{C^{1+a}}\norm{n}^{-a}.
\end{align*}
 Here in $(1)$ we use (i) of Lemma \ref{cor:2} and in $(2)$ we use \eqref{for:21}.

We conclude that for any $v_{i,j}$, $1\leq j\leq N_i$, we have
\begin{align}\label{for:22}
 |\widehat{((\mathcal{R}_i^*)_{v_{i,j}})}_{n}|\leq K_{a}\norm{\Psi_i}_{C^{1+a}}\norm{n}^{-a},\qquad \forall\,n\in M.
\end{align}

Finally, using Lemma \ref{le:4} and \eqref{for:22}, we obtain that for any $n\in M$
\begin{align}\label{for:5}
 |\widehat{(\mathcal{R}_i^*)}_n|&\leq K\sum_{j=1}^{N_i}|\widehat{((\mathcal{R}_i^*)_{v_{i,j}})}_{n}|\norm{n}^N\leq K_{a}\norm{\Psi_i}_{C^{1+a}}\norm{n}^{N-a}\notag\\
 & \leq K_{a}\norm{\Psi_i}_{C^{1+a}}\norm{n}^{N}.
\end{align}

Now for any $r>N+2$ and any $n\in M$ we can estimate
 splitting the exponent of the first term as $\a$ and $1-\a$ in the way to get the total
 the exponent of $\|n\| $ be zero
\begin{align*}
 |\widehat{(\mathcal{R}_i^*)}&_n|\norm{n}^{N+2}=|\widehat{(\mathcal{R}_i^*)}_n|^{\frac{l-2-N}{l+N}}|\widehat{(\mathcal{R}_i^*)}_n|^{\frac{2N+2}{l+N}}\norm{n}^{N+2}\\
 &\overset{\text{(1)}}{\leq} \big(K_{a}\norm{\Psi_i}_{C^{1+a}}\norm{n}^{N}\big)^{\frac{l-2-N}{l+N}}\big(\norm{n}^{-l}\norm{\mathcal{R}_i^*}_l\big)^{\frac{2N+2}{l+N}}\norm{n}^{N+2}\\
 &=K_{a}^{\frac{l-2-N}{l+N}}(\norm{\Psi_i}_{C^{1+a}})^{\frac{l-2-N}{l+N}}(\norm{\mathcal{R}_i^*}_l)^{\frac{2N+2}{l+N}}\\
 &\overset{\text{(2)}}{\leq} K_{l,a}(\norm{\Psi_i}_{C^{1+a}})^{\frac{l-2-N}{l+N}}(\norm{\mathcal{R}_i^*}_{C^l})^{\frac{2N+2}{l+N}}\\
 & \overset{\text{(3)}}{\leq} K_{l,a}(\norm{\Psi_i}_{C^{1+a}})^{\frac{l-2-N}{l+N}}(\norm{\mathcal{R}_i}_{C^{l+\sigma}})^{\frac{2N+2}{l+N}}.
 \end{align*}
Here in $(1)$ we use that $\mathcal{R}_i^*$ is $C^\infty$ and \eqref{for:5}; in $(2)$ we use \eqref{for:19}; in $(3)$ we use
\eqref{for:23}. Then by \eqref{for:19}  we get
\begin{align*}
 \norm{\mathcal{R}_i^*}_{C^0}\leq C\norm{\mathcal{R}_i^*}_{N+2}\leq K_{l,a}(\norm{\Psi_i}_{C^{1+a}})^{\frac{l-2-N}{l+N}}(\norm{\mathcal{R}_i}_{C^{l+\sigma}})^{\frac{2N+2}{l+N}}.
\end{align*}
Finally, we denote by $\Phi$ the map with coordinate maps $\mathcal{R}_i^*$.

%%%%%%%%%%%%  Convergence %%%%%%%%%%%%%%%%
%%%%%%%%%%%%%%%%%%%%%%%%%%%%%%%%%%%%%
%\newpage
\section{Proof of Theorem \ref{th:4} } \label{proof th:4}

 In this section we complete the proof of Theorem \ref{th:4} using an iterative process.
The main part is the inductive step given by Proposition \ref{po:1}. We start with a sufficiently small
perturbation $f_n$ of $A$ which is $C^{1}$ conjugate to $A$.  We construct a smaller perturbation
$f_{n+1}$ which is smoothly conjugate to $f_n$. The conjugacy $\tilde{H}_{n+1}$ between $f_{n}$ and $f_{n+1}$ is obtained using Theorem \ref{th:3}. Then the iterative process is set up so that
$f_n$ converges to $A$ and $\tilde{H}_{1}\circ \cdots\circ \tilde{H}_{n+1}$ converge
in sufficiently high regularity.
%%%%%%%%%%%.  Iterative step %%%%%%%%%%%%%%

\subsection{Iterative step and error estimate} $\;$

We recall the following results, which will be used the proof of Proposition \ref{po:1}.
\begin{lemma}  \cite[Propositions 5.5]{dlLO} \label{dlLO lemma}
For any $r\ge 1$ there exists a constant  $M_r$ such that for any $h,g \in C^r(\M)$,
$$
  \| h\circ g\|_{C^r} \le M_r \left(1+ \| g\|_{C^1}^{r-1} \right)
  \left( \|h\|_{C^1}\| g\|_{C^r} + \|h\|_{C^r}\| g\|_{C^1} \right) +  \|h\|_{C^0} .
  $$
\end{lemma}

\begin{lemma}\label{le:5} \cite[Lemma AII.26.]{La} There is $d>0$ and such that for any $h\in C^r(\M)$, if $\norm{h-I}_{C^1}\leq d$ then $h^{-1}$ exists
with the estimate $\norm{h^{-1}-I}_{C^r}\leq K_r\norm{h-I}_{C^r}$.

\end{lemma}
\begin{proposition}\label{po:1}
Let $A$ be
a weakly irreducible Anosov automorphism of $\T^N$. Let $\beta=\frac{\beta_0}{2}$, where $\beta_0$ is
as in Theorem \ref{HolderConjugacy}.  There exists $0<c<\frac{1}{2}$ such that
for any $C^{\infty}$ perturbation $f_n$ of $A$ satisfying
$$\norm{f_n-A}_{C^{\sigma+2}}<c, \text{  where $\sigma$ is from  Lemma \ref{le:2}},$$
and the conjugacy equation
\begin{align}\label{for:31}
  H_n\circ f_n=A\circ H_n \;\text{ with a function $H_n\in C^1(\TT^N)$ with $\|{H}_{n}-I\|_{C^0}\leq c$}
\end{align}
the following holds.
There exists $\omega_{n+1}\in C^\infty(\TT^N)$ so that the functions
\begin{equation}\label{functions}
\tilde{H}_{n+1}=I-\omega_{n+1}, \quad
 H_{n+1}=H_n\circ \tilde{H}_{n+1},  \quad  f_{n+1}=\tilde{H}_{n+1}^{-1}\circ f_n\circ \tilde{H}_{n+1}\\
\end{equation}
satisfy the new conjugacy equation
\begin{align*}
  H_{n+1}\circ f_{n+1}=A\circ H_{n+1},
\end{align*}
and we have the following estimates.
\begin{itemize}
  \item[\bf{(i)}] For any $r\geq0$ and any $t>1$
  \begin{align*}
 \norm{\omega_{n+1}}_{C^r}\leq K_r\min\{t^\sigma\norm{R_n}_{C^{r}}, \norm{R_n}_{C^{r+\sigma}}\}.
  \quad\text{where}\quad R_n=f_n-A.
\end{align*}

\medskip
 \item[\bf{(ii)}] \label{for:77} For the new error $R_{n+1}=f_{n+1}-A $, we have
  \begin{align*}
    \norm{R_{n+1}}&_{C^0}\leq Kt^\sigma\norm{R_n}_{C^1}\norm{R_n}_{C^{0}}+K_\ell t^{-\ell}\norm{R_n}_{C^\ell}\\
 &+K_{l,\ell}(t^{-\ell+2}\norm{R_n}_{C^\ell}+\norm{R_n}^{1+\frac{\beta}{2}}_{C^2})^{\frac{l-2-N}{l+N}}(t^\sigma\norm{R_n}_{C^{l}})^{\frac{2N+2}{l+N}}
 %  &+C_{l}(t^{-l+1}\norm{R_n}_{C^l}+\norm{R_n}^{1+\frac{\beta_1}{2}}_{C^2})^{\frac{l-2-N}{l}}(t^\sigma\norm{R_n}_{C^{l}})^{\frac{N+2}{l}}
  \end{align*}
 % ? where $\beta_1=\frac{\beta}{(1+\beta)(N+2)}$
 for any $t>1$, $\ell\geq0$ and $l>N+2$; and also for any $r\geq0$ we have
  \begin{align}\label{for:10}
  \norm{R_{n+1}}_{C^r}\leq K_rt^\sigma\norm{R_n}_{C^{r}}+K_r.
\end{align}

\item[\bf{(iii)}] %\label{for:77} - same as \item[\bf{(ii)}]
For the new conjugacy $H_{n+1}$, we have
  \begin{align}\label{for:24}
    \norm{H_{n+1}-I}&_{C^0}\leq K\norm{R_n}_{C^\sigma}+\norm{H_{n}-I}_{C^0}
  \end{align}

\end{itemize}
\end{proposition}

\begin{proof}
%{\color{red} Should we note how $c$ is chosen? I think the choice of $c$ is clear from the proof? }
We denote $h_n=H_n-I$ and $R_n=f_n-A$ and, similarly to  \eqref{conj h}, we write the conjugacy
equation \eqref {for:31} as
\begin{align*}
  Ah_n-h_n\circ f_n=R_n
\end{align*}
%\label{H_*}
%? and decompose it as
%\begin{align*}  A_i(h_n)_i-(h_n)_i\circ f_n=(R_n)_i, \qquad  i=1, \dots , L. \end{align*}
We can assume that $c<\delta$, where $\delta=\delta(\beta)$ is from  Theorem \ref{HolderConjugacy},
and that $\|{H}_{n}-I\|_{C^0}\leq c$ yields that $H$ is the conjugacy close to the identity.
Then Theorem \ref{HolderConjugacy} gives the estimate
\begin{align}\label{for:32}
\norm{h_n}_{C^{1+\beta}}\leq K\norm{R_n}_{C^{1+\beta}}  .
\end{align}
We define
\begin{align}\label{for:34}
 \Omega_{n}=Ah_n-h_n\circ A, \quad\text{and}\quad  \Theta_{n}=  R_n-\Omega_{n} = h_n\circ A-h_n\circ f_n.
\end{align}

\vskip.2cm
\begin{lemma} \label{le:7}
$\|\Theta_n \|_{C^{1+\frac \beta 2}}\le K_A  \,\norm{R_n}^{1+\frac{\beta}{2}}_{C^{1+\beta}}$.
\end{lemma}
%{\tt There is estimate $\|\Theta_n \|_{C^{1+\frac \beta 2}}\le K_A \|h\|_{C^{1+ \beta}} \, \| A-f\|_{C^{1+ \beta}} ^{\frac{\beta}{2}}$ which implies the  lemma.}
%Using this we obtain that \begin{align}\label{for:36} \norm{\Theta_{n}}_{C^{1+\frac{\beta}{2}}}&\leq K_A\,\|h_n\|_{C^{1+ \beta}} \,\norm{R_{n}}_{C^{1+\beta}}^{\frac{\beta}{2}}\leq K' \,\norm{R_n}^{1+\frac{\beta}{2}}_{C^2}. \end{align}

\begin{proof} We omit index $n$ in the proof of the lemma. We note that
$$\| R \|_{C^{1+ \beta}}=\| f-A \|_{C^{1+ \beta}}<c<1.$$
Differentiating at $x\in\TT^N$ we get
\begin{align*}
 D\Theta(x)&\overset{\text{*}}{=}Dh(Ax)\circ A-Dh(fx)\circ Df(x)\notag\\
 &=Dh(Ax)\circ A-Dh(fx)\circ A+Dh(fx)\circ (A-Df(x)),
\end{align*}
and hence
\begin{align*}
  \norm{D\Theta}_{C^0}&\leq  \|A\| \, \norm{Dh(Ax)-Dh(fx)}_{C^0}
  +\norm{Dh(fx)\circ DR(x)}_{C^0}\\
  &\leq \|A\| \, \norm{Dh}_{C^{\b}}\norm{R}^\b_{C^0}+\norm{Dh}_{C^{0}}\norm{DR}_{C^0}\\
  &\leq  \|A\| \,\norm{h}_{C^{1+\b}}\norm{R}^\b_{C^0}+\norm{h}_{C^{1}}\norm{R}_{C^1}.
\end{align*}
Since we also have $ \norm{\Theta}_{C^0} \leq \norm{h}_{C^1}\norm{R}_{C^0}$, we conclude
using \eqref{for:32} and $\| R \|_{C^{1+ \beta}}<1$ that
\begin{align} \label{C^1}
 \norm{\Theta}_{C^1}\leq \|A\| \, \norm{h}_{C^{1+\b}}\norm{R}^\b_{C^0}+\norm{h}_{C^{1}}\norm{R}_{C^1}
 \leq K \norm{R}_{C^{1+\beta}}^{1+\beta}.
\end{align}

To estimate the H\"older norm of $D\Theta$,  using equation $*$ for any $x,\,y\in\TT^N$ we have
\begin{align*}
  &D\Theta(x)-D\Theta(y)\\
  &=\big(Dh(Ax)-Dh(Ay)\big)\circ A+Dh(fx)\circ \big(Df(y)-Df(x)\big)\\
  &+\big(Dh(fy)-Dh(fx)\big)\circ Df(y),
\end{align*}
and hence
\begin{align*}
  &\norm{D\Theta(x)-D\Theta(y)}\\
  &\leq \|A\| \,\norm{Dh(Ax)-Dh(Ay)}+\norm{Dh(fx)}\norm{Df(y)-Df(x)}\\
  &+\norm{Dh(fy)-Dh(fx)}\norm{Df(y)}\\
  &\leq \|A\| \,\norm{Dh}_{C^\b}\norm{Ax-Ay}^\beta+\norm{h}_{C^1}\norm{Df}_{C^{{\b}}}\norm{y-x}^\b\\
  &+\norm{f}_{C^{1}}\norm{Dh}_{C^\b}\norm{fx-fy}^\beta\\
  &\leq \|A\|^{1+\b} \,\norm{h}_{C^{1+\b}}\norm{x-y}^\beta+\norm{h}_{C^1}\norm{f}_{C^{1+\b}}\norm{y-x}^\b\\
  &+\norm{f}_{C^{1}}\norm{h}_{C^{1+\b}}\norm{f}^\beta_{C^{1}}\norm{x-y}^\b.
\end{align*}
We conclude using \eqref{for:32} and $\| f-A \|_{C^{1+ \beta}}<1$ that
\begin{align} \label{C^1H}
  \norm{D\Theta}_{C^{0,\b}}\leq  \|A\|^{1+\b} \, \norm{h}_{C^{1+\b}}+\norm{h}_{C^1}\norm{f}_{C^{1+\b}}+\norm{h}_{C^{1+\b}}\norm{f}^{1+\b}_{C^{1}} \le K \norm{R}_{C^{1+\beta}}.
\end{align}
Therefore
\begin{align}\label{for:35}
 \norm{\Theta}_{C^{1+\beta}}\leq\norm{\Theta}_{C^1}+\norm{D\Theta}_{C^{0,\beta}}\le
 2K \norm{R}_{C^{1+\beta}}.
\end{align}
Finally, we complete the proof of the lemma using an interpolation inequality
\begin{align}\label{for:36}
  \norm{\Theta}_{C^{1+\frac{\beta}{2}}}&\leq K\norm{\Theta}_{C^1}^{\frac{1}{2}}\norm{\Theta}_{C^{1+\beta}}^{\frac{1}{2}}\leq K_A\norm{R}^{1+\frac{\beta}{2}}_{C^{1+\beta}}.
 \end{align}

\end{proof}

We recall that there exists a collection of smoothing operators $\mathfrak{s}_t$, $t>0$, such that for any $s\geq s_1\geq0$ and $s_2\geq0$, for any $g\in C^s(\TT^N)$ the following holds, see \cite{Damjanovic4} and \cite{Hamilton}:
\begin{align} \label{for:28}
 \norm{\mathfrak{s}_tg}_{C^{s+s_2}}\leq K_{s,s_2}\,t^{s_2}\,\norm{g}_{C^{s}},\quad \text{and} \quad
 \norm{(I-\mathfrak{s}_t)g}_{C^{s-s_1}}\leq K_{s,s'}\,t^{-s_1}\,\norm{g}_{C^{s}}.
\end{align}
We write \eqref{for:34} as
\begin{align}\label{for2}
Ah_n-h_n\circ A =\Omega_{n}=R_n-\Theta_{n}=[\mathfrak{s}_{t}R_n]+[(I-\mathfrak{s}_{t})R_n-\Theta_{n}]=:\mathcal R +\Psi
\end{align}
and apply Theorem \ref{th:3} to get the new splitting and obtain the estimates:
\begin{align}\label{for:6}
 \mathfrak{s}_{t}R_n=A\omega_{n+1}-\omega_{n+1}\circ A+\Phi_n
\end{align}
where $\omega_{n+1}$ and $\Phi_n$ are $C^\infty$ maps with the estimates:
\begin{align} %\label{for:24}
 \norm{\omega_{n+1}}_{C^r}&\leq K_r\norm{\mathfrak{s}_{t}(R_n)}_{C^{r+\sigma}}\overset{(a)}{\leq} K_r\min\{t^\sigma\norm{R_n}_{C^{r}}, \norm{R_n}_{C^{r+\sigma}}\},\quad\text{and}\label{for:63}\\
 \norm{\Phi_{n}}_{C^0}&\leq K_{l}(\norm{(I-\mathfrak{s}_{t})R_n-\Theta_{n}}_{C^{1+\frac{\beta}{2}}})^{\frac{l-2-N}{l+N}}(\norm{\mathfrak{s}_{t}R_n}_{C^{l+\sigma}})^{\frac{2N+2}{l+N}}\notag\\
 %  &+C_{l}(\norm{(I-\mathfrak{s}_{t})R_n-\Theta_{n}}_{C^{\frac{\beta_1}{2}},\Delta_2})^{\frac{l-2-N}{l}}(\norm{\mathfrak{s}_{t}R_n}_{C^{l+\sigma}})^{\frac{2N+2}{l}}\notag\\
   &\overset{(b)}{\leq} K_{l}(\norm{(I-\mathfrak{s}_{t})R_n}_{C^{2}}+\norm{R_n}^{1+\frac{\beta}{2}}_{C^2})^{\frac{l-2-N}{l+N}}(\norm{\mathfrak{s}_{t}R_n}_{C^{l+\sigma}})^{\frac{2N+2}{l+N}}\notag\\
 %  &+C_{l,1}(\norm{(I-\mathfrak{s}_{t})R_n}_{C^{1}}+\norm{R_n}^{1+\frac{\beta_1}{2}}_{C^2})^{\frac{l-2-N}{l}}(\norm{\mathfrak{s}_{t}R_n}_{C^{l+\sigma}})^{\frac{2N+2}{l}}\notag\\
   &\overset{(a)}{\leq} K_{l,\ell}(t^{-\ell+2}\norm{R_n}_{C^\ell}+\norm{R_n}^{1+\frac{\beta}{2}}_{C^2})^{\frac{l-2-N}{l+N}}(t^\sigma\norm{R_n}_{C^{l}})^{\frac{2N+2}{l+N}}\label{for:50}
%   &+C_{l,2}(t^{-l+1}\norm{R_n}_{C^l}+\norm{R_n}^{1+\frac{\beta_1}{2}}_{C^2})^{\frac{l-2-N}{l}}(t^\sigma\norm{R_n}_{C^{l}})^{\frac{2N+2}{l}}
\end{align}
for any $r,\,\ell\geq0$ and any $l>N+2$. Here in $(a)$ we use \eqref{for:28} and in $(b)$ we use \eqref{for:36}.

From equation \eqref{for:6} we obtain a $C^r$ estimate for $\Phi_{n}$ with  $r\geq0$
\begin{align*}
 \norm{\Phi_{n}}_{C^r}&=\norm{A\omega_{n+1}-\omega_{n+1}\circ A-\mathfrak{s}_{t}R_n}_{C^r}\\
 &\leq K\norm{\omega_{n+1}}_{C^r}+\norm{\mathfrak{s}_{t}R_n}_{C^r}\overset{(1)}{\leq} K_r t^\sigma\norm{R_n}_{C^{r}}.
\end{align*}
Here in $(1)$ we use \eqref{for:28}  and \eqref{for:63}.

%Then we get \eqref{for:24}.
Let $\tilde{H}_{n+1}=I-\omega_{n+1}$. From \eqref{for:63} we can assume that  $\norm{\omega_{n+1}}_{C^1}<\min\{\frac{1}{2},d\}$ (see Lemma \ref{le:5}) if $c$ is sufficiently small.
Hence $\tilde{H}_{n+1}$ is invertible. We estimate the new error
\begin{align*}
 R_{n+1}=f_{n+1}-A
\end{align*}
 by using
\begin{gather*}
 f_{n+1}=\tilde{H}_{n+1}^{-1}\circ f_n\circ \tilde{H}_{n+1}\Rightarrow \tilde{H}_{n+1}\circ f_{n+1}=f_n\circ \tilde{H}_{n+1}\\
 \Rightarrow (I-\omega_{n+1})\circ f_{n+1}=f_n\circ \tilde{H}_{n+1}\\
 \Rightarrow f_{n+1}=\omega_{n+1}\circ f_{n+1}+f_n\circ \tilde{H}_{n+1}.
\end{gather*}
This gives
\begin{align*}
 R_{n+1}&=\omega_{n+1}\circ f_{n+1}+f_n\circ \tilde{H}_{n+1}-A\\
 &=\omega_{n+1}\circ f_{n+1}+(R_n+A)\circ (I-\omega_{n+1})-A\\
 &=\omega_{n+1}\circ f_{n+1}+R_n\circ (I-\omega_{n+1})-A\circ \omega_{n+1}.
\end{align*}
Hence we see that $R_{n+1}$ has three parts:
\begin{align*}
 R_{n+1}&=\underbrace {\big( {\omega_{n+1}\circ f_{n+1}-\omega_{n+1}\circ A} \big)}_{\mathcal{E}_1}+\underbrace {\big( {R_n\circ (I-\omega_{n+1})-R_n} \big)}_{\mathcal{E}_2}\\
 &+\underbrace {\big( \omega_{n+1}\circ A-{A\circ \omega_{n+1}+R_n} \big)}_{\mathcal{E}_3}.
\end{align*}
We note that
\begin{align*}
 \norm{\mathcal{E}_1}_{C^0}&\leq \norm{\omega_{n+1}}_{C^1}\norm{f_{n+1}-A}_{C^0}\overset{(0)}{\leq} \text{\small$\frac{1}{2}$}\norm{R_{n+1}}_{C^0},\\
 \norm{\mathcal{E}_2}_{C^0}&\leq K\norm{R_n}_{C^1}\norm{\omega_{n+1}}_{C^0}\overset{(1)}{\leq} Kt^\sigma\norm{R_n}_{C^1}\norm{R_n}_{C^{0}};
\end{align*}
and
\begin{align*}
 \norm{\mathcal{E}_3}_{C^0}&=\norm{\Phi_n+(I-\mathfrak{s}_{t})R_n}_{C^0}\leq \norm{\Phi_n}_{C^0}+\norm{(I-\mathfrak{s}_{t})R_n}_{C^0}\\
 &\overset{(2)}{\leq} \norm{\Phi_n}_{C^0}+K_\ell t^{-\ell}\norm{R_n}_{C^\ell}
\end{align*}
for any $\ell\geq0$. Here in $(0)$ we recall that $\norm{\omega_{n+1}}_{C^1}<\frac{1}{2}$; in $(1)$ we use \eqref{for:63}; and in $(2)$ we use \eqref{for:28}.

Hence it follows that
\begin{align*}
 \norm{R_{n+1}}_{C^0}&\leq \norm{\mathcal{E}_1}_{C^0}+\norm{\mathcal{E}_2}_{C^0}+\norm{\mathcal{E}_3}_{C^0}\leq \text{\small$\frac{1}{2}$}\norm{R_{n+1}}_{C^0}+\norm{\mathcal{E}_2}_{C^0}+\norm{\mathcal{E}_3}_{C^0},
\end{align*}
which gives
\begin{align*}
 \norm{R_{n+1}}_{C^0}&\leq 2\norm{\mathcal{E}_2}_{C^0}+2\norm{\mathcal{E}_3}_{C^0}\\
 &\leq  Kt^\sigma\norm{R_n}_{C^1}\norm{R_n}_{C^{0}}+K_rt^{-\ell}\norm{R_n}_{C^\ell}+\norm{\Phi_n}_{C^0}\\
 &\overset{(3)}{\leq} Kt^\sigma\norm{R_n}_{C^1}\norm{R_n}_{C^{0}}+K_\ell t^{-\ell}\norm{R_n}_{C^\ell}\\
 &+K_{l,\ell}(t^{-\ell+2}\norm{R_n}_{C^\ell}+\norm{R_n}^{1+\frac{\beta}{2}}_{C^2})^{\frac{l-2-N}{l+N}}(t^\sigma\norm{R_n}_{C^{l}})^{\frac{2N+2}{l+N}}\notag
  % &+C_{l}(t^{-l+1}\norm{R_n}_{C^l}+\norm{R_n}^{1+\frac{\beta_1}{2}}_{C^2})^{\frac{l-2-N}{l}}(t^\sigma\norm{R_n}_{C^{l}})^{\frac{2N+2}{l}},
\end{align*}
for any $l>N+2$.  Here in $(3)$ we use \eqref{for:50}.
\vskip.2cm

Now we estimate  $\norm{R_{n+1}}_{C^r}$. We note that
\begin{align*}
 R_{n+1}=(I-\omega_{n+1})^{-1}\circ (R_n+A)\circ (I-\omega_{n+1})-A=(I-\omega_{n+1})^{-1}\circ P-A.
\end{align*}
By Lemma \ref{dlLO lemma} we have
\begin{align*}
 \norm{P}_{C^r}&\leq M_r \left(1+ \| I-\omega_{n+1}\|_{C^1}^{r-1} \right)\\
  &\cdot\left( \|R_n+A\|_{C^1}\| I-\omega_{n+1}\|_{C^r} + \|R_n+A\|_{C^r}\| I-\omega_{n+1}\|_{C^1} \right) +  \|R_n+A\|_{C^0}\\
  &\overset{(1)}{\leq} K_rt^\sigma\norm{R_n}_{C^{r}}+K_r,\qquad\text{and}\qquad
  \norm{P}_{C^1}\overset{(1)}{\leq}  K.
\end{align*}
Here in $(1)$ we use the fact that $\omega_{n+1}$ satisfies the estimates  $\norm{\omega_{n+1}}_{C^r}\leq K_rt^\sigma\norm{R_n}_{C^{r}}$ (see \eqref{for:63}) and $\norm{\omega_{n+1}}_{C^1}<\frac{1}{2}$. Using Lemma \ref{le:5} this also implies that
\begin{align*}
 \norm{(I-\omega_{n+1})^{-1}}_{C^r}\leq K_r\norm{\omega_{n+1}}_{C^{r}}\leq K_{r,1}t^\sigma\norm{R_n}_{C^{r}}
\end{align*}
 and $\norm{(I-\omega_{n+1})^{-1}}_{C^1}<2$.

As a direct consequence of Lemma \ref{dlLO lemma} and the above discussion we have
\begin{align*}
 \norm{R_{n+1}}_{C^r}&\leq M_r \left(1+ \| P\|_{C^1}^{r-1} \right)
  \left( \|(I-\omega_{n+1})^{-1}\|_{C^1}\| P\|_{C^r} + \|(I-\omega_{n+1})^{-1}\|_{C^r}\| P\|_{C^1} \right)+K\\
  &\leq K_r\| P\|_{C^r}+K_rt^\sigma\norm{R_n}_{C^{r}}+K\\
  &\leq K_{r,1}t^\sigma\norm{R_n}_{C^{r}}+K_{r,1}.
\end{align*}
 To get \eqref{for:24} we have
 \begin{align*}
    \norm{H_{n+1}-I}_{C^0}&=\norm{H_{n}\circ (I-\omega_{n+1})-I}_{C^0}\leq \norm{H_{n}\circ (I-\omega_{n+1})-H_n}_{C^0}+\norm{H_{n}-I}_{C^0}\\
    &\leq \norm{H_n}_{C^1}\norm{\omega_{n+1}}_{C^0}+\norm{H_{n}-I}_{C^0}\\
    &\overset{(1)}{\leq} K\norm{R_n}_{C^\sigma}+\norm{H_{n}-I}_{C^0}.
  \end{align*}
Here in $(1)$ we use \eqref{for:32} and \eqref{for:63}.
\end{proof}

%%%%%%%%%%%%%%%%%%%%%%%%%%%%%%%%
%%%%%%%%%%%%%%%%%%%%%%%%%%%%%%%

\subsection{The iteration scheme}
First we note that by \cite[Theorem 6.1]{L1} there exists $\sigma_0=\sigma_0(A)\in \N$
such that if $H$ and $H^{-1}$ are $C^{\sigma_0}$ then $H$ and $H^{-1}$ are $C^{\infty}$.

To set up the iterative process we take $\ell$ sufficiently large so that the following holds
\begin{equation} \label{for:78}
\begin{aligned}
&\ell\geq \max\left\{ {\text {\small $\frac{3\sigma+10}{1-\frac{\beta}{3}},\,\,\frac{24\sigma}{\beta},\,\,2(5\max\{\sigma_0,\sigma\}+1)$}},\,\,2(2\sigma+5)\right\}, \\
&{\text {\small $\left(1+\frac{\beta}{2}\right) \left(1-\frac{5}{\ell}\right) \left(\frac{\ell-2-N}{\ell+N}\right)-2\frac{2N+2}{\ell+N}$}}\,\geq\, {\text {\small $1+\frac{\beta}{3}$}}.
\end{aligned}
\end{equation}
Now we construct $R_{n}$, $f_{n}$, $\omega_n$ and $H_{n}$ inductively as follows. For $n=0$ we take
\begin{align*}
 f_0=f,\quad H_0=H,\quad R_0=f-A,\quad \omega_0=0, \quad \text {and define } \;\epsilon_n=\epsilon^{\gamma^n}
\end{align*}
where $\gamma=1+\frac{\beta}{4}$ and $\epsilon>0$ is sufficiently small so that the following holds
\begin{align*}
 \norm{R_0}_{C^0} \leq \epsilon_0=\e ,\qquad \norm{R_0}_{C^\ell}\leq \epsilon_0^{-1},\qquad  \quad \|H_0-I\|_{C^0}<\epsilon_0^{\frac{1}{2}}.
\end{align*}
We note that $H_0\in C^1(\TT^N)$ by Theorem \ref{HolderConjugacy}. Now we assume inductively
that $H_n\in C^1(\TT^N)$ satisfies the conjugacy equation
\begin{align*}
  H_n\circ f_n=A\circ H_n
\end{align*}
and that $H_n$ and $R_n=f_n-A$ satisfy
\begin{align*}
 \norm{R_n}_{C^0}\leq \epsilon_n,\qquad \norm{R_n}_{C^\ell}\leq \epsilon_n^{-1},
\quad \|H_n-I\|_{C^0}<\sum_{i=0}^{n-1}\epsilon^{\frac{1}{2}}_i .
\end{align*}
By interpolation inequalities we have
\begin{align}
\norm{R_n}_{C^{\sigma+2}}\leq K_\ell\norm{R_n}^{\text{\tiny$\frac{\ell-2-\sigma}{\ell}$}}_{C^0}\norm{R_n}^{\text{\tiny$\frac{2+\sigma}{\ell}$}}_{C^{\ell}}<\epsilon^{1-\frac{5+2\sigma}{\ell}}_n\leq \epsilon^{\frac{1}{2}}_n. \label{for:8}
\end{align}
provided $\ell\geq2(2\sigma+5)$.  Here, and subsequently, we estimate various constants
from above by $\e_n^{-\frac 1 \ell}$. This can be done since $\ell$ is fixed, we can take $\epsilon$ small enough.
We also have
\begin{align}\label{for:25}
 \|H_n-I\|_{C^0}<\sum_{i=0}^{n-1}\epsilon^{\frac{1}{2}}_i<\sum_{i=1}^{\infty}(\epsilon^{\frac{1}{4}})^i<2\epsilon^{\frac{1}{4}}.
\end{align}
Then \eqref{for:8} and \eqref{for:25} allow us to use Proposition \ref{po:1} to obtain
the new iterates $R_{n+1}$, $f_{n+1}$, $\omega_{n+1}$ and $H_{n+1}$.
Now we show that these iterates satisfy the inductive assumption and establish appropriate convergence.

\subsection{Inductive estimates and convergence} $\;$

\noindent We use Proposition \ref{po:1} with $t_n=\epsilon_n^{-\frac{3}{\ell}}$ and $l=\ell$.
\vskip.15cm

$(1)\,$ $C^\ell$ \emph{estimate for} $R_{n+1}$ \begin{align*}
 \norm{R_{n+1}}_{C^\ell}&\leq K_\ell t_n^\sigma\norm{R_n}_{C^{\ell}}+K_\ell\leq K_\ell \epsilon_n^{-\frac{3\sigma}{\ell}}(\epsilon_n^{-1}+1)\\
 &<\epsilon_n^{-1-\frac{\beta}{8}-\frac{3\sigma}{\ell}}\leq \epsilon_n^{-1-\frac{\beta}{4}}=\epsilon_{n+1}^{-1},
\end{align*}
provided $\ell\geq\frac{24\sigma}{\beta}$.

\vskip.15cm
\smallskip
$(2)\,$ $C^0$ \emph{estimate for} $R_{n+1}$
\begin{align*}
    \norm{R_{n+1}}_{C^0}&\leq Kt_n^\sigma\norm{R_n}_{C^2}^2+K_\ell t_n^{-\ell}\norm{R_n}_{C^\ell}\\
    &+K_{\ell}(t_n^{-\ell+2}\norm{R_n}_{C^\ell}+\norm{R_n}^{1+\frac{\beta}{2}}_{C^2})^{\frac{\ell-2-N}{\ell+N}}(t_n^\sigma\norm{R_n}_{C^{\ell}})^{\frac{2N+2}{\ell+N}}\notag\\
   &\overset{\text{(a)}}{\leq} K\epsilon_n^{\text{\tiny$2-\frac{3\sigma+10}{\ell}$}}+K_\ell\epsilon_n^{3}\epsilon_n^{-1}\\
    &+K_\ell(\epsilon_n^{\frac{3(\ell-2)}{\ell}}
    \epsilon_n^{-1}+\epsilon_n^{\text{\tiny$(1+\frac{\beta}{2})(1-\frac{5}{\ell})$}})^{\frac{\ell-2-N}{\ell+N}}(\epsilon_n^{-\frac{3\sigma}{\ell}}
    \epsilon_n^{-1})^{\frac{2N+2}{\ell+N}}\\
   &\overset{\text{(b)}}{\leq} K\epsilon_n^{\text{\tiny$2-\frac{3\sigma+10}{\ell}$}}+K_\ell\epsilon_n^{2}\\
    &+2K_\ell(\epsilon_n^{\text{\tiny$(1+\frac{\beta}{2})(1-\frac{5}{\ell})$}})^{\frac{\ell-2-N}{\ell+N}}(
    \epsilon_n^{-2})^{\frac{2N+2}{\ell+N}}\\
    &\overset{\text{(c)}}{<} \epsilon_n^{\gamma}=\epsilon_{n+1}.
   \end{align*}
  Here in $(a)$ we use interpolation inequalities:
  \begin{align}\label{for:68}
\norm{R_n}_{C^{2}}\leq C \norm{R_n}^{\text{\tiny$\frac{\ell-2}{\ell}$}}_{C^0}\norm{R_n}^{\text{\tiny$\frac{2}{\ell}$}}_{C^{\ell}}
<\epsilon_n^{\text{\tiny$1-\frac{5}{\ell}$}};
\end{align}
in $(b)$  we note that
\begin{align*}
  \text{\small$(1+\frac{\beta}{2})(1-\frac{5}{\ell})$}<\text{\small$2(1-\frac{5}{\ell})$}<2-\frac{6}{\ell}\quad\text{and}\quad\frac{3\sigma}{\ell}<1.
\end{align*}
Then $\epsilon_n^{-\frac{3\sigma}{\ell}}
    \epsilon_n^{-1}<\epsilon_n^{-2}$ and
\begin{align*}
 \max\{\epsilon_n^{\text{\tiny$(1+\frac{\beta}{2})(1-\frac{5}{\ell})$}}, \epsilon_n^{\frac{3(\ell-2)}{\ell}}
    \epsilon_n^{-1}\}= \epsilon_n^{\text{\tiny$(1+\frac{\beta}{2})(1-\frac{5}{\ell})$}};
\end{align*}
in $(c)$ we use
\begin{gather*}
 \epsilon_n^{\text{\tiny$2-\frac{3\sigma+10}{\ell}$}}<\epsilon_n^{1+\text{\tiny$\frac{\beta}{3}$}},\quad (\epsilon_n^{\text{\tiny$(1+\frac{\beta}{2})(1-\frac{5}{\ell})$}})^{\frac{\ell-2-N}{\ell+N}}(
    \epsilon_n^{-2})^{\frac{2N+2}{\ell+N}}<\epsilon_n^{1+\text{\tiny$\frac{\beta}{3}$}},
    \end{gather*}
provided
\vskip.15cm
$\hskip2cm  2-\frac{3\sigma+10}{\ell}\geq 1+\frac{\beta}{3},\qquad (1+\frac{\beta}{2})(1-\frac{5}{\ell})(\frac{\ell-2-N}{\ell+N})-2\frac{2N+2}{\ell+N}\geq 1+\frac{\beta}{3}$.

\vskip.15cm \noindent
By \eqref{for:78} and the assumption all inequalities above are satisfied.

\vskip.2cm

$(3)$\, $C^{\sigma_0}$ \emph{estimate for} $\omega_{n+1}$: By interpolation inequalities we have
\begin{align*}
\norm{R_{n}}_{C^{\sigma_0}}\leq K_\ell\norm{R_{n}}^{\text{\tiny$\frac{\ell-\sigma_0}{\ell}$}}_{C^0}\norm{R_{n}}^{\text{\tiny$\frac{\sigma_0}{\ell}$}}_{C^{\ell}}<\epsilon_n^{\text{\tiny$1-\frac{2\sigma_0+1}{\ell}$}}.
\end{align*}
Hence we have
\begin{align}\label{for:7}
 \norm{\omega_{n+1}}_{C^{\sigma_0}}\leq Kt_n^\sigma\norm{R_n}_{C^{\sigma_0}}\leq K\epsilon_n^{-\frac{3\sigma}{\ell}}\epsilon_n^{\text{\tiny$1-\frac{2\sigma_0+1}{\ell}$}}
 <\epsilon_{n}^{\frac{1}{2}},\end{align}
provided
\begin{align*}
 -\frac{3\sigma}{\ell}+1-\frac{2\sigma_0+1}{\ell}>\frac{1}{2},
\end{align*}
which is satisfied for $\ell>2(5\max\{\sigma_0,\sigma\}+1)$.

\vskip.2cm

$(4)$\, $C^{0}$ \emph{estimate for} $H_{n+1}$: By \eqref{for:8} we have
\begin{align*}
    \norm{H_{n+1}-I}&_{C^0}\leq K\norm{R_n}_{C^\sigma}+\norm{H_{n}-I}_{C^0}<K\epsilon^{1-\frac{5+2\sigma}{\ell}}_n+\sum_{i=0}^{n-1}\epsilon^{\frac{1}{2}}_n\leq \epsilon^{\frac{1}{2}}_n+\sum_{i=0}^{n-1}\epsilon^{\frac{1}{2}}_i=\sum_{i=0}^{n}\epsilon^{\frac{1}{2}}_i
  \end{align*}

\medskip
Consequently, we have
\begin{align*}
 f_{n+1}&=\tilde{H}_{n+1}^{-1}\circ\tilde{H}_{n}^{-1}\circ\cdots\circ \tilde{H}_{1}^{-1}\circ f\circ \tilde{H}_{1}\circ \cdots\circ \tilde{H}_{n+1}\\
 &=\mathfrak{L}_{n+1}^{-1}\circ f\circ\mathfrak{L}_{n+1}
\end{align*}
where $\tilde{H}_{i}=I-\omega_{i}$, $1\leq i\leq n+1$; and $\mathfrak{L}_{n+1}=\tilde{H}_{1}\circ \cdots\circ \tilde{H}_{n+1}$.

Finally,  \eqref{for:7} implies that $\mathfrak{L}_n$ converges in $C^{\sigma_0}$ topology to a $C^{\sigma_0}$  diffeomorphism $H$, which is a conjugacy between $f$ and $A$. By \cite[Theorem 6.3]{L1} and the choice of ${\sigma_0}$ we conclude that $H$ is a $C^{\infty}$ diffeomorphism.

%%%%%%%%%%%%%%%% BIBLIOGRAPHY %%%%%%%%%%%%%
%%%%%%%%%%%%%%%%%%%%%%%%%%%%%%%%%%%%%%

\end{document}